\documentclass[11pt]{article}

\usepackage[utf8]{inputenc} 

\usepackage{amsmath}
\usepackage{amssymb}
\usepackage{amsthm}

\usepackage[all]{xy}

\usepackage{appendix}
\usepackage{rotating}

\usepackage{hyperref}

\newcommand{\weg}[1]{}

\newcommand{\C}{{\mathbb{C}}}
\newcommand{\R}{{\mathbb{R}}}
\newcommand{\Z}{{\mathbb{Z}}}
\newcommand{\N}{{\mathbb{N}}}
\newcommand{\K}{{\mathfrak{K}}}

\newcommand{\cfrak}{{\mathfrak{c}}}
\newcommand{\cbar}{{\overline{\mathfrak{c}}}}
\newcommand{\kfrak}{{\mathfrak{uc}}}
\newcommand{\kbar}{{\overline{\mathfrak{uc}}}}

\newcommand{\FC}{{\mathcal{O}(V,\mathcal{F})}}
\newcommand{\fol}{{(V,\mathcal{F})}}

\DeclareMathOperator{\id}{id}
\DeclareMathOperator{\im}{im}

\DeclareMathOperator{\supp}{supp}

\theoremstyle{theorem}
\newtheorem{thm}{Theorem}[section]
\newtheorem{lem}[thm]{Lemma}

\newtheorem{prop}[thm]{Proposition}
\newtheorem{cor}[thm]{Corollary}
\newtheorem{axiom}[thm]{Axiom}

\newtheorem*{mainthm1}{Theorem \ref{mainthm}}
\newtheorem*{mainthm2}{Theorem \ref{mainthmcoeff}}

\theoremstyle{definition}
\newtheorem{example}[thm]{Example}
\newtheorem{mdef}[thm]{Definition}

\newtheorem*{acknow}{Acknowledgements}

\theoremstyle{remark}

\newtheorem*{msc}{Mathematics Subject Classification (2010)}
\newtheorem*{keywords}{Keywords}

\title{Coarse co-assembly as a ring homomorphism}
\author{Christopher Wulff}
\date{} 

\begin{document}

\maketitle

\begin{abstract}
The $K$-theory of the stable Higson corona of a coarse space carries a canonical ring structure. This ring is the domain of an unreduced version of the coarse co-assembly map of Emerson and Meyer. 
We show that the target also carries a ring structure and co-assembly is a ring homomorphism, provided that the given coarse space is contractible in a coarse sense.
\end{abstract}

\begin{msc}
19K35, 46L80.
\end{msc}
\begin{keywords}
Coarse geometry, co-assembly, ring structures.
\end{keywords}

\section*{Introduction}
Let $X$ be a countably generated coarse space, e.\,g.\ a proper metric space.
Emerson and Meyer constructed in \cite{EmeMey} a coarse co-assembly map 
$$\mu^*:\tilde K_{1-*}(\cfrak(X))\to KX^*(X)$$
which is dual to the usual coarse assembly map (cf.\ \cite{HigRoe,RoeITCGTM})
$$\mu: KX_*(X)\to K_*(C^*(X))$$
in the sense that there are natural pairings 
$$KX^*(X)\times KX_*(X)\to \Z,\qquad \tilde K_{1-*}(\cfrak(X))\times K_*(C^*X)\to \Z$$
such that 
$$\langle x,\mu(y)\rangle=\langle \mu^*(x),y\rangle\qquad \forall x\in\tilde K_{1-*}(\cfrak(X)),y\in KX_*(X).$$
The domain of the coarse co-assembly map is the reduced $K$-theory of the stable Higson corona $\cfrak(X)$, which is the $C^*$-algebra of all continuous functions of vanishing variation on $X$ into the $C^*$-algebra of compact operators $\K$ modulo the ideal $C_0(X,\K)$.

At first one would expect $K^*(C^*X)$ -- the $K$-homology of the Roe algebra $C^*X$ -- as the domain of $\mu^*$. However, it is not even clear how to define this  group, because $C^*X$ lacks separability. In contrast $\tilde K_{1-*}(\cfrak(X))$ behaves in many ways as expected. For instance we have functoriality under coarse maps and $\mu^*$ is an isomorphism for scalable spaces which are uniformly contractible and have bounded geometry.

This paper is concerned with an unreduced version of the coarse co-assembly map,
$$\mu^*:K_{*}(\cfrak(X))\to KX^{1-*}(X\setminus\{pt\}),$$
which we derive from the co-assembly map of Emerson and Meyer.
The right hand side is by definition the $K$-theory of the Rips complex $\mathcal{P}$ of $X$ with one point $pt\in\mathcal{P}$ removed.

There is a canonical graded ring structure on the left hand side given by multiplication of functions in $\cfrak(X)$ by means of a $*$-homomorphism $\K\otimes\K\cong \K$. It is a natural question to ask whether there is also a secondary ring structure
$$KX^{i}(X\setminus\{pt\})\otimes  KX^{j}(X\setminus\{pt\})\to  KX^{i+j-1}(X\setminus\{pt\})$$
 on the right hand side such that $\mu^*$ becomes a ring homomorphism.
The usual primary ring structure in $K$-theory,
 $$KX^{i}(X\setminus\{pt\})\otimes  KX^{j}(X\setminus\{pt\})\to  KX^{i+j}(X\setminus\{pt\}),$$
 is obviously not suitable because of degree reasons.

The most enlightening example is the open cone $X=\mathcal{O}Y$ over a nice compact base space $Y$. In this case, the topological side can be identified with $K^{-*}(Y)$ and co-assembly becomes a ring isomorphism
$$K_*(\cfrak(X))\cong K^{-*}(Y).$$

It turns out that an  appropriate prerequisite  for the existence of a secondary ring structure in the general case  is the notion of coarse contractibility which we introduce in Definition  \ref{coarsecontraction}. Open cones over compact spaces are coarsely contractible in this sense. So are  foliated cones as defined in \cite{RoeFoliations}, $\operatorname{CAT}(0)$ spaces and hyperbolic metric spaces, in particular Gromov's hyperbolic groups.

Given a coarse contraction, we obtain a suitable contraction $H:\mathcal{P}\times[0,1]\to\mathcal{P}$ onto the point $pt$ which we use to construct the proper continuous map
\begin{align*}
\Gamma:\,(0,1)\times (\mathcal{P}\setminus\{pt\})&\to (\mathcal{P}\setminus\{pt\})\times (\mathcal{P}\setminus\{pt\})
\\(t,x)&\mapsto\begin{cases}(H(x,1-2t),x)&t\leq1/2\\(x,H(x,2t-1))&t\geq1/2.\end{cases}
\end{align*}
Given this data, the secondary product on $KX^{*}(X\setminus\{pt\}):=K^*(\mathcal{P}\setminus\{pt\})$ is defined as $(-1)^i$ times the composition
\begin{align*}
K^i(\mathcal{P}\setminus\{pt\})\otimes K^j(\mathcal{P}\setminus\{pt\})&\xrightarrow[\hspace{2.5ex}]{}
K^{i+j}(\,(\mathcal{P}\setminus\{pt\})\times(\mathcal{P}\setminus\{pt\})\,)
\\&\xrightarrow[\hspace{2.5ex}]{\Gamma^*} K^{i+j}(\,(0,1)\times (\mathcal{P}\setminus\{pt\})\,)
\\&\xrightarrow[\hspace{2.5ex}]{\cong} K^{i+j-1}(\mathcal{P}\setminus\{pt\}).
\end{align*}
Note that this secondary product arises like most secondary products by comparing two different reasons for the vanishing of the primary product.
Here the two reasons are the homotopies $\Gamma|_{(0,1/2]\times (\mathcal{P}\setminus\{pt\})}$ and  $\Gamma|_{[1/2,1)\times (\mathcal{P}\setminus\{pt\})}$ from the diagonal embedding to ``infinity''.

Our main result is the following.
\begin{mainthm1}
Let $X$ be a coarsely contractible countably generated coarse space. Then the unreduced coarse co-assembly map 
$$\mu^*:K_{*}(\cfrak(X))\to KX^{1-*}(X\setminus\{pt\})$$
is a ring homomorphism.
\end{mainthm1}

There is also a more general coarse co-assembly map with coefficients in a $C^*$-algebra $D$. 
The stable Higson corona $\cfrak(X;D)$ with coefficients in $D$ is defined just like $\cfrak(X)$ with the only difference that we consider functions $X\to D\otimes\K$. Furthermore, if we define
$$KX^{*}(X\setminus\{pt\};D):=K_{-*}(C_0(\mathcal{P}\setminus\{pt\})\otimes D)$$
then there are products
\begin{align*}
K_i(\cfrak(X;D))\otimes K_j(\cfrak(X;E))&\to K_{i+j}(\cfrak(X;D\otimes E))
\\KX^{i}(X\setminus\{pt\};D)\otimes KX^{j}(X\setminus\{pt\};E)&\to KX^{i+j-1}(X\setminus\{pt\};D\otimes E)
\end{align*}
which are interior products with respect to the space but exterior products with respect to the coefficient algebra. Co-assembly also respects these products.
\begin{mainthm2}
Let $X$ be a coarsely contractible countably generated coarse space. Then the unreduced coarse co-assembly map with coefficients in $D$,
$$\mu^*:K_{*}(\cfrak(X;D))\to KX^{1-*}(X\setminus\{pt\};D)\,,$$
is multiplicative.
\end{mainthm2}
In fact, using the stabilized version of the Higson corona is not necessary. Everything works just as well for the usual Higson corona or the unstabilized Higson corona with coefficients. 
However, the work of Emerson and Meyer indicates that it is the stabilized versions which one needs to consider in applications.

Finally, it should be noted that some conjectures of Roe in \cite{RoeFoliations} may be reformulated and proved by replacing the groups $K^*_{FJ}(V/\mathcal{F}):=K^{*+1}(C^*(\FC^\to))$ defined therein by the rings $K^*(\cfrak(\FC))$, where the coarse space $\FC$ is a foliated cone constructed from the foliation $\fol$. This will be the topic of a forthcoming paper.

This paper is organized as follows. Section \ref{sec:CoarseGeometry} recalls basic notions of coarse geometry and introduces Higson type corona $C^*$-algebras.
In the subsequent three sections \ref{sec:sigmaspaces}--\ref{sec:KTheory} we give an account of the basic facts about $\sigma$-coarse and $\sigma$-locally compact spaces as well as $\sigma$-$C^*$-algebras and their $K$-theory. This is necessary as the Rips complex is not a locally compact space. Properties of $K$-theory of $\sigma$-$C^*$-algebras concerning the products which cannot be found in the literature are proven in Appendices \ref{sec:Frechet}, \ref{sec:kkTheory}.

In a nutshell, sections \ref{sec:sigmaspaces}--\ref{sec:KTheory} say that these notions  behave very much like their non-$\sigma$-counterparts. Thus, the impatient reader should get away with skimming over these sections. 

In Section \ref{sec:coassembly} we review the definition of the Rips complex and the concept of coarse co-assembly before deriving unreduced versions from the reduced ones. The $K$-theory product of the stable Higson corona is introduced in Section \ref{sec:coronaringstructure} where we also establish co-assembly as ring isomorphism in the case of open cones.
The secondary ring structure on the target is defined in Section \ref{sec:secondaryproduct}, where we also prove basic properties of this product. Multiplicativity  of the co-assembly map is proven in Section \ref{sec:multiplicativity}. Finally, we introduce the notion of coarse contractibility as a prerequisite for $\sigma$-contractibility of the Rips complex in Section \ref{sec:coarsecontractability}.

One last comment: in this paper we shall use the symbol $\otimes$ for the \emph{maximal} tensor product of $\sigma$-$C^*$-algebras.

\begin{acknow}
This paper represents part of the research project for a doctoral dissertation at the University of Augsburg.
The author would like to thank his thesis advisor Bernhard Hanke for his steady encouragement and advice.
Furthermore, the author is grateful to Alexander Engel for his comments on a first draft of this paper and useful discussions.
Part of this work was carried out during a research stay at the Max Planck Institute for Mathematics in Bonn, whose hospitality is gratefully acknowledged.
The doctoral project was supported by a grant of the Studienstiftung des Deutschen Volkes.
\end{acknow}

\tableofcontents


\section{Coarse geometry and coronas}\label{sec:CoarseGeometry}
The purpose of this first section is to recall some basics of coarse geometry and Higson corona $C^*$-algebras.
Less surprisingly, as far as coarse geometry is concerned, we shall stick quite closely to the very concise presentation of this topic in \cite[Section 2]{EmeMey}. 
More comprehensive references to coarse geometry are \cite{RoeCoarseGeometry} and \cite[Chapter 6]{HigRoe}.

\begin{mdef}[{\cite[Definition 1]{EmeMey}}]
A \emph{coarse structure} on a set $X$ is a collection $\mathcal{E}$ of subsets of $X\times X$, called controlled sets or entourages, which contains the diagonal and is closed under the formation of subsets, inverses, products and finite unions and contains all finite subsets, i.\,e.\ 
\begin{enumerate}\itemsep0pt\parskip0pt 
\item $\Delta_X:=\{(x,x)\,|\,x\in X\}\in \mathcal{E}$;
\item if $E\in\mathcal{E}$ and $E'\subset E$, then $E'\in\mathcal{E}$;
\item if $E\in\mathcal{E}$, then $E^{-1}:=\{(y,x)\,|\,(x,y)\in E\}\in\mathcal{E}$;
\item if $E_1,E_2\in\mathcal{E}$, then 
$$E_1\circ E_2:=\{(x,z)\,|\,\exists y\in X:(x,y)\in E_1\text{ and } (y,z)\in E_2\}\in \mathcal{E};$$
\item if $E_1,E_2\in\mathcal{E}$, then $E_1\cup E_2\in \mathcal{E}$;
\item if $E\subset X\times X$ is finite, then $E\in\mathcal{E}$.
\end{enumerate}
A subset $B\subset X$ is called \emph{bounded}, if $B\times B\in \mathcal{E}$.
A \emph{coarse space} 
 is a locally compact space $X$ equipped with a coarse structure $\mathcal{E}$ such that in addition
\begin{enumerate}\itemsep0pt\parskip0pt \setcounter{enumi}{6}
\item  some neighborhood of the diagonal $\Delta\subset X\times X$ is an entourage;
\item every bounded subset of $X$ is relatively compact.
\end{enumerate}

If $X$ is a coarse space and $A\subset X$, then the \emph{subspace coarse structure} is the set of all subsets of $A\times A$ which are entourages of $X$. If $A$ is closed in $X$, then $A$ is a coarse space, too.

If $X,Y$ are coarse spaces with coarse structures $\mathcal{E}_X,\mathcal{E}_Y$, then the \emph{product coarse structure} on $X\times Y$ consists of all subsets of $(X\times Y)\times (X\times Y)$ which are contained in 
$$E_X\times E_Y:=\{(x_1,y_1,x_2,y_2)\,|\,(x_1,x_2)\in E_X,\,(y_1,y_2)\in E_2\}$$
for some $E_X\in\mathcal{E}_X$, $E_Y\in \mathcal{E}_Y$. Again, $X\times Y$ is a coarse space.

A coarse space is called \emph{countably generated} if there is an increasing sequence $(E_n)_{n\in\N}$ of entourages such that any entourage is contained in $E_n$ for some $n\in\N$.
\end{mdef}
\begin{example}
Let $(X,d)$ be a metric space in which every bounded set is relatively compact. The \emph{metric coarse structure} on $X$ is the smallest coarse structure such that the sets 
$$E_R:=\{(x,y)\in X\times X\,|\,d(x,y)\leq R\},\quad R\in \N$$
are entourages. Equipped with this coarse structure, $X$ becomes a countably generated coarse space.
\end{example}
\begin{mdef}
A \emph{coarse map} $\phi:X\to Y$ between coarse spaces is a Borel map such that $\phi\times\phi$ maps entourages to entourages and $\phi$ is proper in the sense that preimages of bounded sets are bounded. Two coarse maps $\phi,\psi:X\to Y$ are called \emph{close} if $(\phi\times\psi)(\Delta_X)\subset Y\times Y$ is an entourage.

The \emph{coarse category of coarse spaces} is the category of coarse spaces and closeness classes of coarse maps between them. A coarse map is called a \emph{coarse equivalence} if it is an isomorphism in this category.
\end{mdef}

We are now in a position to introduce Higson type corona $C^*$-algebras which are the main coarse geometric players in this paper.
\begin{mdef} \label{coarsefunctionalgebras}
Let $X$ be a coarse space, $pt\in X$ a distinguished point, $(Y,d)$ a metric space and $D$ any $C^*$-algebra.
\begin{enumerate}
\item A Borel map $f:X\to Y$
is said to have \emph{vanishing variation}, if for any entourage $E\subset X\times X$ the function
$$\operatorname{Var}_Ef:\,X\to [0,\infty),\quad x\mapsto\sup\{d(f(x),f(y))\,|\,(x,y)\in E\}$$
vanishes at infinity. \cite[Definition 7]{EmeMey}
\item For any coarse space $X$ and any $C^*$-algebra $D$, we let $\kbar(X;D)$ be the $C^*$-algebra of bounded, continuous functions of vanishing variation $X\to D$. It is called the \emph{unstable Higson compactification of $X$ with coefficients $D$}.
\item The \emph{unstable Higson corona of $X$ with coefficients $D$} is the quotient $C^*$-algebra $\kfrak(X;D):=\kbar(X;D)/C_0(X;D)$.
\item The \emph{pointed unstable Higson compactification of $X$ with coefficients $D$} is
$$\kbar_0(X;D):=\{f\in\kbar(X;D)\,|\,f(pt)=0\}.$$
We have $\kfrak(X;D)=\kbar_0(X;D)/C_0(X\setminus\{pt\};D)$.
\item The stable counterparts of these function algebras are obtained by replacing $D$ with $D\otimes\K$:
\begin{align*}
\cbar(X;D)&:=\kbar(X;D\otimes\K),
\\\cbar_0(X;D)&:=\kbar_0(X;D\otimes\K),
\\\cfrak(X;D)&:=\kfrak(X;D\otimes\K).
\end{align*}
In particular, $\cfrak(X;D)$ is the \emph{stable Higson corona of $X$ with coefficients $D$}.  \cite[Definition 8]{EmeMey}
\item If $D=\C$, we usually omit $D$ from notation.
\end{enumerate}
\end{mdef}
\begin{prop}
The assignments $X\mapsto \kfrak(X,D),\, X\mapsto \cfrak(X,D),\, X\mapsto \cfrak(X),$ are contravariant functors from the coarse category of coarse spaces to the category of $C^*$-algebras.

The assignments $X\mapsto \kbar(X,D),\, X\mapsto \cbar(X,D),\, X\mapsto \cbar(X),$ are contravariantly functorial with respect to continuous coarse maps.

The assignments $X\mapsto \kbar_0(X,D),\, X\mapsto \cbar_0(X,D),\, X\mapsto \cbar_0(X),$ are contravariantly functorial with respect to pointed continuous coarse maps.

Furthermore, there is the obvious covariant functoriality in the coefficient algebra.
\end{prop}
\begin{proof}
See \cite[Proposition 13]{EmeMey} for the nontrivial parts of this proposition.
\end{proof}


\section{{$\sigma$-locally compact spaces and $\sigma$-coarse spaces}}\label{sec:sigmaspaces}
\begin{mdef}[{\cite[Section 2]{EmeMey}}]
A $\sigma$-locally compact space is an increasing sequence $(X_n)_{n\in \N}$ of subsets of a set $\mathcal{X}$ such that 
\begin{itemize}\itemsep0pt\parskip0pt 
\item $\mathcal{X}=\bigcup_{n\in \N}X_n$,
\item each $X_n$ is a locally compact Hausdorff space,
\item $X_n$ carries the subspace topology of $X_m$ for all $n\leq m$,
\item $X_n$ is closed in $X_m$ for all $n\leq m$.
\end{itemize}
\end{mdef}
By the usual abuse of notation, we will often call $\mathcal{X}$ instead of the sequence $(X_n)_{n\in \N}$ a $\sigma$-locally compact space. However, the sequence is an essential part of the definition.

Locally compact spaces are $\sigma$-locally compact spaces with all $X_n$ equal.

The set $\mathcal{X}$ can and will always be endowed with the final topology:
A subset $\mathcal{A}\subset\mathcal{X}$ is open/closed if all the intersections $A_n=\mathcal{A}\cap X_n$ are open/closed.
If $\mathcal{A}\subset\mathcal{X}$ is open or closed, then $\mathcal{X}\setminus\mathcal{A}=\bigcup_{n\in\N}X_n\setminus A_n$ is again a $\sigma$-locally compact space.

Note, that the following notion of morphism and cartesian product gives $\sigma$-locally compact spaces the structure of  a monoidal category.
\begin{mdef}\label{sigmaspacesmonoidalcategory}
Let $\mathcal{X}=\bigcup_{n\in\N}X_n$ and $\mathcal{Y}=\bigcup_{n\in\N}Y_n$ be  $\sigma$-locally compact spaces.
\begin{itemize}
\item A morphism $f: \mathcal{X}\to \mathcal{Y}$ is a map $f:\mathcal{U}\to\mathcal{Y}$ from an open subset $\mathcal{U}\subset\mathcal{X}$ such that for each $m\in\N$ there is $n\in\N$ such that $f(U_m)\subset Y_n$ and the restriction $f|_{U_m}:U_m\to Y_n$ is a proper continuous map.\footnote{These morphisms correspond to pointed continuous maps between one point compactifications.}

We will also call such morphisms simply \emph{proper continuous maps}.
\item The cartesian product is defined by the sequence $(X_m\times Y_m)_{m\in\N}$. Passing forgetfully to the category of sets we see that
$$\mathcal{X}\times\mathcal{Y}=\bigcup_{m\in\N}X_m\times Y_m$$
and we apply the usual abuse of notation of denoting the product simply by $\mathcal{X}\times\mathcal{Y}$ instead of $(X_m\times Y_m)_{m\in\N}$.
\end{itemize}
\end{mdef}
It is worth noticing that continuity of the maps $f|_{U_m}:U_m\to Y_n$ is equivalent to continuity of $f:\mathcal{U}\to\mathcal{Y}$ in the final topologies.

A homotopy between two proper continuous maps $f,g:\mathcal{X}\to\mathcal{Y}$ is of course a proper continuous map $\mathcal{X}\times[0,1]\to\mathcal{Y}$ restricting to $f,g$ at $0,1$, respectively.

Given a $\sigma$-locally compact space $\mathcal{X}=\bigcup_{n\in \N}X_n$, a $C^*$-algebra $D$ and a closed subset $\mathcal{A}\subset\mathcal{X}$, we can define the following function algebras:
\begin{align*}
C_0(\mathcal{X};D)&=\{f:\mathcal{X}\to D:\,f|_{X_n}\in C_0(X_n;D)\,\forall n\in \N\}
\\C_b(\mathcal{X},\mathcal{A};D)&=\{f:\mathcal{X}\to D:\,f|_{X_n}\in C_b(X_n;D)\,\forall n\in \N,\,f|_{\mathcal{A}}=0\}
\end{align*}
Note that they are $\sigma$-$C^*$-algebras  in the sense of Definition \ref{defsigmaCstar} below, with $C^*$-seminorms given by $p_n(f)=\|f|_{X_n}\|_\infty$. Our definition of morphisms in the category of $\sigma$-locally compact spaces is such that a proper continuous map 
$f:\mathcal{X}\to\mathcal{Y}$ induces a $*$-homomorphism $f^*:\,C_0(\mathcal{Y};D)\to C_0(\mathcal{X};D)$.
If additionally $f$ is defined on all of $\mathcal{X}$ and $\mathcal{A}\subset\mathcal{X},\mathcal{B}\subset\mathcal{Y}$ are closed subsets such that $f(\mathcal{A})\subset\mathcal{B}$, then $f$ induces a $*$-homomorphism $f^*:\,C_b(\mathcal{Y},\mathcal{B};D)\to C_b(\mathcal{X},\mathcal{A};D)$.

More properties of these function algebras are discussed in the next section.

We proceed with $\sigma$-coarse spaces.
\begin{mdef}
A $\sigma$-coarse space is a $\sigma$-locally compact space, in which each $X_n$ is a coarse space and $X_n$ has the subspace coarse structure of $X_m$ for all $n\leq m$ \cite[Section 2]{EmeMey}.
A  \emph{coarse continuous map} is a map $f:\mathcal{X}\to\mathcal{Y}$ such that  for each $m\in\N$ there is $n\in\N$ such that $f(X_m)\subset Y_n$ and the restriction $f|_{X_m}:X_m\to Y_n$ is a coarse continuous map.
\end{mdef}
Given a $\sigma$-coarse space $\mathcal{X}=\bigcup_{n\in \N}X_n$ and a $C^*$-algebra $D$, we define the $\sigma$-$C^*$-algebra
$$\kbar(\mathcal{X};D):=\{f:\mathcal{X}\to D:\,f|_{X_n}\in \kbar(X_n;D)\,\forall n\in \N\}.$$
In the same manner we can generalize the other function algebras of Definition \ref{coarsefunctionalgebras}.
In particular, if $pt\in X_0$ is a point then
\begin{align*}
\kbar_0(\mathcal{X};D):&=\{f:\mathcal{X}\to D:\,f|_{X_n}\in \kbar_0(X_n;D)\,\forall n\in \N\}
\\&=\{f\in\kbar(\mathcal{X};D)\,:\,f(pt)=0\},
\\\kfrak(\mathcal{X};D):&=\kbar(\mathcal{X};D)/C_0(\mathcal{X};D)=\kbar_0(\mathcal{X};D)/C_0(\mathcal{X}\setminus\{pt\};D).
\end{align*}
A coarse continuous map $f:\mathcal{X}\to\mathcal{Y}$ mapping the basepoint in $X_0$ to the basepoint in $Y_0$ induces a $*$-homomorphism $f^*:\,\kbar_0(\mathcal{Y};D)\to \kbar_0(\mathcal{X};D)$.


\section{{$\sigma$-$C^*$-algebras}}\label{sec:sigmaalgebras}
A good exposition of $\sigma$-$C^*$-algebras can be found in \cite{PhiInv}. In this section we summarize all the properties that we shall need. 
\begin{mdef}[{c.f. \cite{PhiInv}}]\label{defsigmaCstar}
A $\sigma$-$C^*$-algebra is a complex topological $*$-algebra such that its topology
\begin{itemize}\itemsep0pt\parskip0pt 
\item is Hausdorff,
\item is generated by a countable family of  $C^*$-seminorms $(p_\alpha)_{\alpha\in D}$, i.e.\ the $p_\alpha$ are submultiplicative and satisfy the $C^*$-identity
$$p_\alpha(a^*a)=p_\alpha(a)^2\quad\forall a\in A,$$
\item is complete with respect to the family of $C^*$-seminorms.
\end{itemize}
\end{mdef}
We can (and will) always assume without loss of generality that the index set is $D=\N$ and the $C^*$-seminorms are an increasing sequence: $p_n\leq p_m\;\forall n\leq m$

The quotients $A_n=A/\ker(p_n)$  are $C^*$-algebras and the continuous $*$-homo\-mor\-phisms $A\to A_n$ and $A_m\to A_n$ ($n\leq m$) are surjective.
Furthermore, $A\cong\varprojlim_n  A_n$ in the category of topological $*$-algebras.\footnote{The algebra underlying this topological algebra is the inverse limit of the algebras underlying the $A_n$.} In fact, we have the following equivalent characterization:
\begin{prop}[{\cite[Section 5]{PhiInv}}]
A topological $*$-algebra is a $\sigma$-$C^*$-algebra if and only if it is an inverse limit of an inverse system of  $C^*$-algebras indexed over the natural numbers.
\end{prop}
As we have seen, the $*$-homomorphisms in this inverse system may be chosen to be surjections.

The $\sigma$-$C^*$-algebras defined in the previous section are the inverse limits of the inverse systems of $C^*$-algebras $C_0(X_n;D),\,C_b(X_n,A_n;D),\,\kbar_0(X_n;D)$. In fact, we could have also chosen a $\sigma$-$C^*$-algebra $D$ as coefficient algebra by defining the function algebras as the inverse limits of $C_0(X_n;D_n)$, $C_b(X_n,A_n;D_n)$, $\kbar_0(X_n;D_n)$. As before, they can also be defined as algebras of functions $\mathcal{X}\to D$.

We now collect some important features of $\sigma$-$C^*$-algebras.
\begin{prop}[{\cite[Theorem 5.2]{PhiInv}}]\label{autocont}
Just as with $C^*$-algebras, $*$-homo\-mor\-phisms between $\sigma$-$C^*$-algebras are automatically continuous.
\end{prop}
Just as in \cite{PhiInv}, we will always call such continuous $*$-homomorphism simply \emph{homomorphisms}.

\emph{Ideals} in $\sigma$-$C^*$-algebras will always be closed two-sided selfadjoint ideals.
\begin{prop}[{\cite[Corollary 5.4]{PhiInv}}]
Let $A$ be a $\sigma$-$C^*$-algebra and let $I$ be an ideal in $A$. Then $A/I$ is a $\sigma$-$C^*$-algebra, and every homomorphism $\varphi:A\to B$ of $\sigma$-$C^*$-algebras such that $\varphi|_I=0$ factors through $A/I$.
\end{prop}
If $\mathcal{X}$ is a $\sigma$-locally compact space, $\mathcal{A}$ a closed subset and $D$ a coefficient $C^*$-algebra, then $C_0(\mathcal{X}\setminus\mathcal{A};D)$ is an ideal in $C_b(\mathcal{X},\mathcal{A};D)$. If $\mathcal{X}$ is a $\sigma$-coarse space with basepoint $pt$ and $D$ a coefficient algebra, then $C_0(\mathcal{X}\setminus\{pt\};D)$ is an ideal in $\kbar_0(\mathcal{X};D)$, so $\kfrak(\mathcal{X};D)=\kbar_0(\mathcal{X};D)/C_0(\mathcal{X}\setminus\{pt\};D)$ is a $\sigma$-$C^*$-algebra.
\begin{lem}[{cf.\ \cite[Lemma 18]{EmeMey}}]\label{coarselyequivalentfiltration}
If all inclusions $X_n\subset X_m$ ($m\leq n$) are coarse equivalences, then $\kfrak(\mathcal{X};D)$ is a $C^*$-algebra isomorphic to $\kfrak(X_n;D)$ for all $n$.
\end{lem}

A sequence 
\begin{equation}\label{exactseq}
0\to I\xrightarrow{\alpha}A\xrightarrow{\beta}B\to 0
\end{equation}
of $\sigma$-$C^*$-algebras and homomorphisms is called \emph{exact} if it is algebraically exact, $\alpha$ is a homeomorphism onto its image, and $\beta$ defines a homeomorphism of $A/\ker(\beta)\to B$.
\begin{prop}\label{exactseqproperties}
\begin{enumerate}
\item For the sequence of $\sigma$-$C^*$-algebras and homomorphisms \eqref{exactseq} 
 to be exact, it is sufficient that it be algebraically exact. \cite[Corollary 5.5]{PhiInv}
\item The sequence of $\sigma$-$C^*$-algebras and homomorphisms \eqref{exactseq} is exact if and only if it is an inverse limit (with surjective maps) of exact sequences of $C^*$-algebras. \cite[Proposition 5.3(2)]{PhiInv}\label{exactseqpropertylim}
\end{enumerate}
\end{prop}
If $\mathcal{X}$ is a $\sigma$-locally compact space and $\mathcal{A}$ a closed subset, then
$$0\to C_0(\mathcal{X}\setminus\mathcal{A})\to C_0(\mathcal{X})\to C_0(\mathcal{A})\to 0$$
is an exact sequence of $\sigma$-$C^*$-algebras. This assignment is natural under proper continuous maps between pairs of spaces $(\mathcal{X},\mathcal{A})\to(\mathcal{Y},\mathcal{B})$, i.e.\ proper continuous maps $\mathcal{X}\to\mathcal{Y}$ mapping the closed subset $\mathcal{A}\subset\mathcal{X}$ to the closed subset $\mathcal{B}\subset\mathcal{Y}$.

Just as for $C^*$-algebras, we can define the unitalization $\tilde A$ of a  $\sigma$-$C^*$-algebra $A$ by extending multiplication and $C^*$-seminorms to $\tilde A=\C\oplus A$. 
Equivalently, if $A$ is written as inverse limit $\varprojlim_n A_n$ of $C^*$-algebras $A_n$, then  $\tilde A=\varprojlim_n \tilde A_n$.
We obtain the exact sequence $$0\to A\to\tilde A\to \C\to 0.$$

The maximal tensor product $\otimes$ of $C^*$-algebras can be extended to $\sigma$-$C^*$-algebras: Let $A,B$ be two $\sigma$-$C^*$-algebras and $p,q$ continuous $C^*$-seminorms on $A,B$, respectively. We denote by $p\otimes q$ the greatest $C^*$-cross-seminorm determined by $p$ and $q$, i.\,e.\ $p\otimes q$ is the greatest $C^*$-seminorm such that $(p\otimes q)(a\otimes b)=p(a)q(b)$ on elementary tensors $a\otimes b$.

\begin{mdef}[{\cite[Definition 3.1]{PhiInv}}]
The maximal tensor product $A\otimes B$ of two $\sigma$-$C^*$-algebras $A,B$ is the completion of their algebraic tensor product $A\odot B$ with respect to the set of all $C^*$-cross-seminorms.

It is again a $\sigma$-$C^*$-algebra, with topology generated by the countable family  of $C^*$-seminorms $(p_n\otimes q_n)_{n\in\N}$.
\end{mdef}
 Of course, there is a universal property similar to the one for the maximal tensor product of $C^*$-algebras:
\begin{prop}[{\cite[Proposition 3.3]{PhiInv}}]\label{tensormaxunivprop}
Let $A,B,C$ be $\sigma$-$C^*$-algebras and let $\varphi:A\to C$ and let $\psi:B\to C$ be two 
homomorphisms whose ranges commute. Then there is a unique 
homomorphism $\eta:A\otimes B\to C$ such that $\eta(a\otimes b)=\varphi(a)\psi(b)$ for all $a\in A,b\in B$.
\end{prop}
Directly from the definition of the maximal tensor product we deduce that
if $f:A_1\to A_2$, $g:B_1\to B_2$ are continuous $*$-homomorphisms, then 
$f\otimes g:A_1\odot B_1\to A_2\odot B_2$
extends to a continuous $*$-homomorphism $A_1\otimes B_1\to A_2\otimes B_2$.

\begin{prop}[{\cite[Proposition 3.2]{PhiInv}}]\label{productinvlim}
If $A=\varprojlim_n A_n$ and $B=\varprojlim_n B_n$, then $A\otimes B=\varprojlim_{n} A_n\otimes B_n$.
\end{prop}
This proposition allows the straightforward generalization of some known facts of $C^*$-algebras:
\begin{cor}
For any $\sigma$-locally compact space $\mathcal{X}=\bigcup_{n\in\N}X_n$ and coefficient $\sigma$-$C^*$-algebra $D$ the equation
$$C_0(\mathcal{X};D)=C_0(\mathcal{X})\otimes D$$
holds.
\end{cor}
\begin{proof}
$C_0(\mathcal{X};D)=\varprojlim_{n\in\N}C_0(X_n;D_n)=\varprojlim_{n\in\N}C_0(X_n)\otimes D_n=C_0(\mathcal{X})\otimes D$.
\end{proof}
\begin{cor}
Let $\mathcal{X}=\bigcup_{n\in\N}X_n$ and $\mathcal{Y}=\bigcup_{n\in\N}Y_n$ be $\sigma$-locally compact spaces. Then 
$$C_0(\mathcal{X}\times \mathcal{Y})=C_0(\mathcal{X})\otimes C_0(\mathcal{Y}).$$
\end{cor}
\begin{proof}
$C_0(\mathcal{X}\times \mathcal{Y})=\varprojlim_nC_0(X_n\times Y_n)=\varprojlim_nC_0(X_n)\otimes C_0(Y_n)=C_0(\mathcal{X})\otimes C_0(\mathcal{Y}).$
\end{proof}

The property of maximal tensor products of $\sigma$-$C^*$-algebra, which is most important to us, is exactness:
\begin{thm}\label{tensorexact}
The maximal tensor product of $\sigma$-$C^*$-algebras is exact.
\end{thm}
\begin{proof}
Given a short exact sequence
$$0\to I\to A\to B\to 0$$
of $\sigma$-$C^*$-algebras, Proposition \ref{exactseqproperties} 
 allows us to write 
$$I=\varprojlim I_n,\quad A=\varprojlim A_n,\quad B=\varprojlim B_n$$
where the directed systems fit into a commutative diagram
$$\xymatrix{
&\vdots\ar[d]&\vdots\ar[d]&\vdots\ar[d]&
\\0\ar[r]&I_{n+1}\ar[d]\ar[r]&A_{n+1}\ar[d]\ar[r]&B_{n+1}\ar[d]\ar[r]&0
\\0\ar[r]&I_{n}\ar[d]\ar[r]&A_{n}\ar[d]\ar[r]&B_{n}\ar[d]\ar[r]&0
\\&\vdots&\vdots&\vdots&
}$$
with exact rows and all vertical maps surjective.

The maximal tensor product of $C^*$-algebras is exact, so if $D=\varprojlim D_n$ is another $\sigma$-$C^*$-algebra, then
$$\xymatrix{
&\vdots\ar[d]&\vdots\ar[d]&\vdots\ar[d]&
\\0\ar[r]&I_{n+1}\otimes D_{n+1}\ar[d]\ar[r]&A_{n+1}\otimes D_{n+1}\ar[d]\ar[r]&B_{n+1}\otimes D_{n+1}\ar[d]\ar[r]&0
\\0\ar[r]&I_{n}\otimes D_{n}\ar[d]\ar[r]&A_{n}\otimes D_{n}\ar[d]\ar[r]&B_{n}\otimes D_{n}\ar[d]\ar[r]&0
\\&\vdots&\vdots&\vdots&
}$$
has exact rows and all vertical maps are surjective. Again by Proposition \ref{exactseqproperties} 
 the sequence
$$0\to \varprojlim I_n\otimes D_n\to \varprojlim A_n\otimes D_n\to \varprojlim B_n\otimes D_n\to 0$$
is exact and the claim follows from Proposition \ref{productinvlim}.
\end{proof}
\begin{cor}
If $I_{1,2}\subset A_{1,2}$ are ideals, then $I_1\otimes I_2\to A_1\otimes A_2$ is a homeomorphism onto its image. Thus, $I_1\otimes I_2$ is an ideal in $A_1\otimes A_2$.
\end{cor}
\begin{cor}\label{tensorquotients}
There are canonical isomorphisms of $\sigma$-$C^*$-algebras 
\begin{align*}
(A_1\otimes A_2)/(I_1\otimes A_2)&\cong A_1/I_1\otimes A_2,
\\(A_1\otimes A_2)/(A_1\otimes I_2)&\cong A_1\otimes A_2/I_2,
\\(A_1\otimes A_2)/(A_1\otimes I_2+I_1\otimes A_2)&\cong A_1/I_1\otimes A_2/I_2.
\end{align*}
\end{cor}
\begin{proof}
The first two are direct consequences of exactness of the tensor product and the definition of short exact sequences of $\sigma$-$C^*$-algebras. For the third, note that algebraically we have
\begin{align*}
\frac{A_1\otimes A_2}{A_1\otimes I_2+I_1\otimes A_2}&\cong \frac{(A_1\otimes A_2)/(I_1\otimes A_2)}{(A_1\otimes I_2+I_1\otimes A_2)/(I_1\otimes A_2)}
\\&\cong \frac{(A_1\otimes A_2)/(I_1\otimes A_2)}{(A_1\otimes I_2)/(I_1\otimes A_2\cap A_1\otimes I_2)}
\\&= \frac{(A_1\otimes A_2)/(I_1\otimes A_2)}{(A_1\otimes I_2)/(I_1\otimes I_2)}\cong \frac{A_1/I_1\otimes A_2}{A_1/I_1\otimes I_2}
\\&\cong A_1/I_1\otimes A_2/I_2.
\end{align*}
Continuity of this isomorphism and its inverse are automatic by Proposition \ref{autocont}.
\end{proof}
This corollary will come in handy for constructing homomorphisms 
$$
A_1/I_1\otimes A_2\to B,\quad A_1\otimes A_2/I_2\to B,\quad A_1/I_1\otimes A_2/I_2\to B
$$
into another $\sigma$-$C^*$-algebra $B$.

The notion of homotopy is of course the canonical one:
\begin{mdef}
A homotopy between two homomorphisms $f,g:\,A\to B$ is a homomorphism $A\to C[0,1]\otimes B$ such that evaluation at $0,1$ yields $f$ and $g$.
\end{mdef}
A homotopy $H:\,[0,1]\times \mathcal{X}\to\mathcal{Y}$ between two proper continuous maps $f,g:\,\mathcal{X}\to\mathcal{Y}$ gives rise to a homotopy
$H^*:\,C_0(\mathcal{Y})\to C[0,1]\otimes C_0(\mathcal{X})$ between $f^*,g^*$.


\section{{$K$-theory of $\sigma$-$C^*$-algebras}}\label{sec:KTheory}
The construction of the coarse co-assembly map in \cite{EmeMey} requires $K$-theory for $\sigma$-$C^*$-algebras. This theory was developed in  \cite{PhiRep} (see also \cite{PhiFre}). To make ourselves independent of a concrete picture of $K$-theory, we state the properties we need in the following five axioms. Their $C^*$-algebraic counterparts are well known.

\begin{axiom}\label{Ktheoryaxiomfunctor}
$K$-theory of $\sigma$-$C^*$-algebras  is a covariant homotopy functor from the category of $\sigma$-$C^*$-algebras into the category of $\Z/2$-graded abelian groups.
\end{axiom}

\begin{axiom}\label{Ktheoryaxiomsequence}
Naturally associated to each exact sequence $0\to I\to A\to B\to 0$ of $\sigma$-$C^*$-algebras is a six term exact sequence
$$\xymatrix{
K_0(I)\ar[r]&K_0(A)\ar[r]&K_0(B)\ar[d]
\\K_1(B)\ar[u]&K_1(A)\ar[l]&K_1(I).\ar[l]
}$$
\end{axiom}

\begin{axiom}\label{Ktheoryaxiomreflection}
Let $\tau:(0,1)\to(0,1),\,t\mapsto 1-t$. For any $\sigma$-$C^*$-algebra $A$, the induced homomorphisms
$$K_*(C_0(0,1)\otimes A)\xrightarrow{(\tau^*\otimes\id_A)_*}K_*(C_0(0,1)\otimes A)$$
is multiplication by $-1$.
\end{axiom}

\begin{axiom}\label{Ktheoryaxiomproduct}
There is an associative and graded commutative exterior product
$$K_i(A)\otimes K_j(B)\to K_{i+j}(A\otimes B)$$
which is natural in both variables.
\end{axiom}

\begin{axiom}\label{Ktheoryaxiomcompatibility}
The exterior product is compatible with boundary maps in the following sense:\label{Ktheorypropertyboundary}
The diagram 
$$\xymatrix{
K_i(B)\otimes K_j(D)\ar[r]\ar[d]& K_{i-1}(I)\otimes K_j(D)\ar[d]
\\K_{i+j}(B\otimes D)\ar[r]&K_{i+j-1}(I\otimes D)
}$$
commutes and the diagram
$$\xymatrix{
K_i(D)\otimes K_j(B)\ar[r]\ar[d]& K_{i}(D)\otimes K_{j-1}(I)\ar[d]
\\K_{i+j}(D\otimes B)\ar[r]&K_{i+j-1}(D\otimes I)
}$$
commutes up to a sign $(-1)^i$ whenever the upper horizontal arrows are connecting homomorphisms associated to a short exact sequence  $0\to I\to A\to B\to 0$ of $\sigma$-$C^*$-algebras, the lower horizontal arrows are the connecting homomorphism associated to the short exact sequence obtained by tensoring the first with another $\sigma$-$C^*$-algebra $D$ and the vertical maps are exterior multiplication.
\end{axiom}

Axioms \ref{Ktheoryaxiomfunctor}--\ref{Ktheoryaxiomreflection} were proved in \cite{PhiRep,PhiFre}. However, those versions of $K$-theory for $\sigma$-$C^*$-algebras are not very well adapted to the construction of products.
We use the products of Cuntz' $kk$-theory \cite{CunBiv}, which is much more general then $K$-theory for $\sigma$-$C^*$-algebras, to construct an exterior product satisfying Axioms \ref{Ktheoryaxiomproduct}, \ref{Ktheoryaxiomcompatibility} in Appendix \ref{sec:kkTheory}. Doing this forces us to work with Fréchet algebras which we shall review in Appendix \ref{sec:Frechet}.

The way we prove these axioms in the appendices is somewhat unsatisfactory as we have to mix up results from different theories. 
It would be nice to have a picture of $K$-theory for $\sigma$-$C^*$-algebras which is well adapted to both products and boundary maps and allows a  direct approach to proving their compatibility as in Axiom  \ref{Ktheoryaxiomcompatibility}.


\section{Unreduced coarse co-assembly}\label{sec:coassembly}
This section gives an overview over the concept of coarse co-assembly. There is a variety of maps which deserve to be called coarse co-assembly maps:
\begin{enumerate}
\item\label{coassemblytype1} Let $X$ be a countably generated, unbounded coarse space and $D$ a $C^*$-algebra. The coarse co-assembly map of Emerson and Meyer \cite{EmeMey} (see Definition \ref{coassemblydefEmeMey} below) is a map
$$\mu^*:\tilde K_{1-*}(\cfrak(X;D))\to KX^*(X;D)\,.$$
Its domain is the reduced $K$-theory of the stable Higson corona,
$$\tilde K_{*}(\cfrak({X},D)):=K_{*}(\cfrak({X},D))/\,\im[K_{*}(D\otimes \K)\xrightarrow[{\rm const.\ fu's}]{\rm incl.\ as} K_{*}(\cfrak({X},D))]\,.$$
Its target $KX^*(X;D):=K_{-*}(C_0(\mathcal{P}_Z)\otimes D)$ is the $K$-theory of the Rips complex $\mathcal{P}_Z$ with coefficients in $D$. We recall the Rips complex construction below. 
In case $D=\C$, the coarse co-assembly map is dual to the coarse assembly map
$$\mu: KX_*(X)\to K_*(C^*(X))$$
($C^*(X)$ is the Roe algebra of $X$) 
in the sense that there are natural pairings 
$$KX^*(X)\times KX_*(X)\to \Z,\qquad \tilde K_{1-*}(\cfrak(X))\times K_*(C^*X)\to \Z$$
such that 
$$\langle x,\mu(y)\rangle=\langle \mu^*(x),y\rangle\qquad \forall x\in\tilde K_{1-*}(\cfrak(X)),y\in KX_*(X).$$
\item\label{coassemblytype2}  If $X$ is a uniformly contractible metric space of bounded geometry then $KX^*(X,D)\cong K_{-*}(C_0(X;D))=:K^*(X;D)$ and the coarse co-assembly map of \ref{coassemblytype1} corresponds to a map
$$\tilde K_{1-*}(\cfrak(X;D))\to K^{*}(X;D).$$
In fact, this map exists for any $\sigma$-coarse space $\mathcal{X}$ instead of $X$. We call it the uncoarsened version of the co-assembly map.
\item\label{coassemblytype3}  There are versions of \ref{coassemblytype1} and \ref{coassemblytype2} where the left hand side displays the unreduced $K$-theory of the stable Higson corona. We compensate this on the right hand side by removing a point of the Rips complex resp.\ the $\sigma$-coarse space. Thus, the unreduced co-assembly maps are
\begin{align*}
K_{1-*}(\cfrak(X;D))&\to KX^*(X\setminus\{pt\};D),
\\K_{1-*}(\cfrak({X};D))&\to K^*({X}\setminus\{pt\};D)
\end{align*}
(see Definition \ref{coassemblydefUnredCoar} for the first one). 
They can be obtained from the reduced versions of \ref{coassemblytype1} and \ref{coassemblytype2} by gluing a ray $\R^+=[0,\infty)$ to the space which acts as a distinct base point at infinity.
\item\label{coassemblytype4}  
The most general co-assembly map defined in Definition \ref{coassemblydefUnredUncUnst} below uses the  unstable Higson corona $\kfrak(X,D)$ instead of the stable Higson corona.
For our purposes it is more convenient to do the calculations with the unstable version, as we do not have to keep track of the compact operators everywhere. Eventually we will always return to the stable algebras by replacing $D$ by $D\otimes \K$.
\end{enumerate}
Details on the co-assembly maps  of \ref{coassemblytype1} and \ref{coassemblytype2} can be found in \cite{EmeMey}.
However, the ring structures considered here work only in the unreduced cases. 
The reason for this is that the ring structure which we shall construct on $K_*(\cfrak(X))$  has a unit, namely the unit of $\Z\cong K_*(\K)\subset K_*(\cfrak(X))$, and exactly this unit is identified with $0$ in $\tilde K_*(\cfrak(X))$.

We will now take a closer look at the objects mentioned above starting with the Rips complex.
We briefly recall its construction as presented in \cite[Section 4]{EmeMey}.

Let $Z$ be a countably generated discrete coarse space. Fix an increasing sequence $(E_n)$ of entourages generating the coarse structure. 
Let $\mathcal{P}_Z$ be the set of probability measures on $Z$ with finite support. This is a simplicial complex whose vertices are the Dirac measures on $Z$. It is given the corresponding topology. The locally finite subcomplexes
$$P_{Z,n}:=\{\mu\in\mathcal{P}_Z\,|\,\supp\mu\times\supp\mu\subset E_n\}$$
are locally compact spaces and constitute a $\sigma$-locally compact space $\mathcal{P}_Z=\bigcup_{n\in\N}P_{Z,n}$. 
The $P_{Z,n}$ may be equipped with the coarse structure generated by the entourages
$$\{(\mu,\nu)\,|\,\supp\mu\times\supp\nu\subset E_m\},\quad m\in\N,$$
giving $\mathcal{P}_Z$ the structure of a $\sigma$-coarse space with all inclusions $Z\to P_{Z,n}$ being coarse equivalences.

\begin{mdef}[{\cite[Definition 22]{EmeMey}}]\label{defKX}
Let $X$ be a countably generated coarse space and let $D$ be a $C^*$-algebra. Let $Z\subset X$ be a countably generated, discrete coarse subspace that is coarsely equivalent to $X$. The \emph{coarse $K$-theory of $X$ with coefficients $D$} is defined as
\begin{equation}\label{originalcca}
KX^*(X;D):=K_{-*}(C_0(\mathcal{P}_Z;D)).
\end{equation}
\end{mdef}
In this definition, $Z$ is equipped with the subspace coarse structure. It exists by \cite[Lemma 4]{EmeMey}.

Furthermore, in the setting of the definition, $\cfrak(\mathcal{P}_Z;D)$ is a $C^*$-algebra isomorphic to $\cfrak(X;D)$ by Lemma \ref{coarselyequivalentfiltration}.
\begin{mdef}[{\cite[Definition 25]{EmeMey}}]\label{coassemblydefEmeMey}
Let $X$ be a countably generated, unbounded coarse space and let $D$ be a $C^*$-algebra. Choose $Z\subset X$ as in the previous definition. The \emph{coarse co-assembly map}
$$\tilde K_{1-*}(\cfrak(X;D))\to KX^*(X;D)$$
is the connecting homomorphism of the short exact sequence
$$0\to C_0(\mathcal{P}_Z;D\otimes\K)\to\cbar(\mathcal{P}_Z;D)\to\cfrak(\mathcal{P}_Z;D)\cong\cfrak(X;D)\to 0.$$
\end{mdef}
We need an unreduced version of coarse co-assembly.   
It  is obtained from the reduced version by implementing an idea of \cite{RoeFoliations}: Simply glue a ray $\R^+=[0,\infty)$ to $X$ which acts as a distinct basepoint at infinity.

Given a countably generated coarse space $X$ with a distinguished point $pt$ we define 
$$X^\to:=X\cup_{pt\sim 0}\R^+.$$
If the coarse structure on $X$ is generated by the sequence of entourages $E_n\subset X\times X$, we can equip $X^\to$ with the coarse structure generated by the sequence of entourages
\begin{align*}
E_n^\to:=E_n&\cup\{(s,t)\in\R^+\times\R^+\,|\,|s-t|\leq n\}
\\&\cup\{(x,t)\in X\times\R^+\,|\,(x,pt)\in E_n \wedge t\leq n\}
\\&\cup\{(t,x)\in \R^+\times X\,|\,(x,pt)\in E_n \wedge t\leq n\}\quad\subset X^\to\times X^\to.
\end{align*}
This coarse structure is obviously independent of the choice of the sequence $(E_n)$ and endows $X^\to$ with the structure of a coarse space.
\begin{lem}
Let $X,Z,D$ be as before.
The coarse co-assembly map of $X^\to$ can be canonically identified with the connecting homomorphism
$$K_{1-*}(\cfrak(X;D))\to K_{-*}(C_0(\mathcal{P}_Z\setminus\{pt\};D))$$
associated to the short exact sequence of $\sigma$-$C^*$-algebras
$$0\to C_0(\mathcal{P}_Z\setminus\{pt\};D\otimes\K)\to\cbar_0(\mathcal{P}_Z;D)\to\cfrak(\mathcal{P}_Z;D)\cong\cfrak(X;D)\to 0.$$
\end{lem}

\begin{proof}
We obviously have $\cfrak(X^\to;D)=\cfrak(X;D)\oplus\cfrak(\R^+;D)$, so
$$K_*(\cfrak(X^\to;D))=K_*(\cfrak(X;D))\oplus K_*(\cfrak(\R^+;D)).$$
Furthermore, the map $K_*(D\otimes \K)\to K_*(\cfrak(\R^+;D))$ is an isomorphism by \cite[Theorem 30]{EmeMey}.
It follows that the composition
\begin{equation}\label{unreductionanalyticside}
K_*(\cfrak(X;D))\xrightarrow{incl.} K_*(\cfrak(X^\to;D))\to \tilde K_*(\cfrak(X^\to;D)).
\end{equation}
is an isomorphism.

We choose $Z\cup_{pt\sim 0}\N\subset X^\to$ as preferred coarsely equivalent discrete subspace to calculate the coarse $K$-theory of $X^\to$. 
Consider the inclusion of simplicial complexes
\begin{equation}\label{addedrayKX}
\mathcal{P}_Z\cup_{pt\sim 0}\mathcal{P}_{\N}\subset \mathcal{P}_{Z\cup_{pt\sim 0}\N}.
\end{equation}
The left hand side inherits the structure of a $\sigma$-coarse space by choosing the obvious filtration by the subcomplexes
$$P_{Z,n}\cup_{pt\sim 0}P_{\N,n}\subset P_{Z\cup_{pt\sim 0}\N,\,n}$$
 and giving each of them the corresponding subspace coarse structure. These inclusions are coarse equivalences for all $n$, because both sides differ only by a finite number of simplices. Furthermore, all of these subcomplexes are coarsely equivalent to $X^\to$. We obtain canonical isomorphisms of $C^*$-algebras
$$\cfrak(\mathcal{P}_Z\cup_{pt\sim 0}\mathcal{P}_{\N};D)\cong \cfrak(\mathcal{P}_{Z\cup_{pt\sim 0}\N};D)\cong \cfrak(X^\to;D)\cong \cfrak(X;D)\oplus \cfrak(\R^+;D).$$

We claim that the inclusion \eqref{addedrayKX}
is also a homotopy equivalence. We shall construct a homotopy inverse as sketched in Figure \ref{fig:Ripscomplexhomotopy}. 
\begin{figure}[htbp]
\begin{center}
\setlength{\unitlength}{0.075\textwidth}
\begin{picture}(10.4,7.4)(-6.2,-2.2)
\multiput(-6,0)(1,0){11}{\circle*{0.1}}

\multiput(1,1)(1,0){4}{\circle*{0.1}}
\multiput(2,2)(1,0){3}{\circle*{0.1}}
\multiput(3,3)(1,0){2}{\circle*{0.1}}
\multiput(2,4)(1,0){3}{\circle*{0.1}}
\multiput(3,5)(1,0){2}{\circle*{0.1}}

\multiput(1,-1)(1,0){4}{\circle*{0.1}}
\multiput(1,-2)(1,0){4}{\circle*{0.1}}

\multiput(0.95,-1.9)(-0.28,0.56){3}{\line(-1,2){0.2}}
\multiput(1.76,3.52)(0.28,0.56){3}{\line(1,2){0.2}}

\multiput(-4.15,0)(-0.5,0){5}{\line(-1,0){0.2}}

\put(3.4,0.4){$Z$}
\put(-5.8,-0.35){$\mathbb{N}$}

\put(-0.3,-0.35){$0\!=\!pt$}
\put(-4,-0.35){$n$}
\put(3,3.7){$x$}
\put(-2.1,0.7){$\kappa>\lambda$}
\put(-0.2,1.7){$\kappa<\lambda$}

\thicklines
\put(0,0){\line(-1,0){4}}
\put(0,0){\line(3,4){3}}
\qbezier(-4,0)(-0.5,2)(3,4)
\qbezier(-0.5,2)(-0.25,1)(0,0)

\thinlines
\put(-3,0.5){\vector(0,-1){0.4}}
\put(-2.5,0.75){\vector(0,-1){0.65}}
\put(-2,0.5){\vector(0,-1){0.4}}

\put(-0.6,1.8){\vector(1,-4){0.42}}
\put(-1.1,1.5){\vector(1,-4){0.35}}
\put(-1.4,0.5){\vector(1,-4){0.1}}

\put(-0.36,1.94){\vector(1,-4){0.42}}
\put(0.1,1.5){\vector(1,-3){0.3}}
\put(0.6,2.5){\vector(1,-2){0.45}}
\put(1.1,2.8){\vector(1,-1){0.48}}
\put(1.9,3.3){\vector(2,-1){0.38}}

\end{picture}
\caption{The simplex of $\mathcal{P}_{Z\cup_{pt\sim 0}\N}$ spanned by the points $n\in \mathbb{N}, x\in Z$ and $0=pt\in \mathbb{N}\cap Z$ is mapped to the union of the simplex  spanned by $n,0$ and the simplex spanned by $pt,x$.}
\label{fig:Ripscomplexhomotopy}
\end{center}
\end{figure}
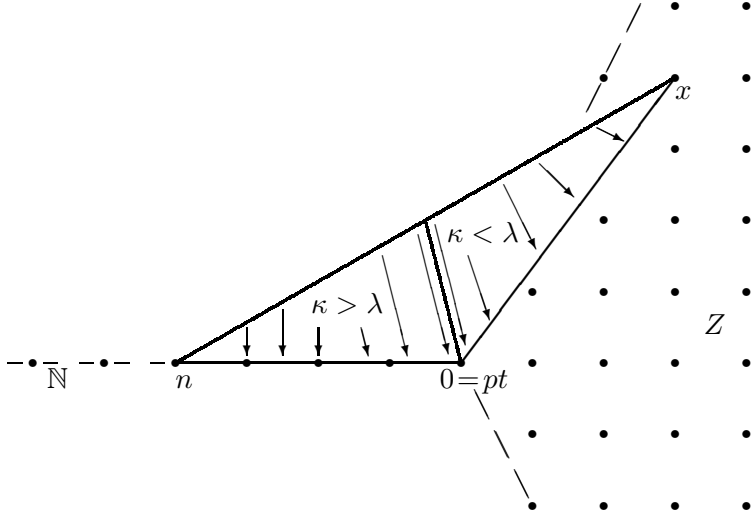

Let
$$\mu=\iota\delta_0+\sum_{n\in\N\setminus\{0\}}\kappa_n\delta_n+\sum_{x\in Z\setminus\{pt\}}\lambda_x\delta_x\quad\in\mathcal{P}_{Z\cup_{pt\sim 0}\N},$$
where $\delta_x$ denotes the Dirac measure with support $x$. The sums are finite, $\iota,\kappa_n,\lambda_x\geq 0$ for all $n\in\N,x\in Z$ and $\iota+\sum\kappa_n+\sum\lambda_x=1$. 
We map $\mu$ to
$$\sum_{n\in\N\setminus\{0\}}\left(\frac{2\kappa}{1-\iota}-1\right)\kappa_n\delta_n+\left(1-\frac{2\kappa^2}{1-\iota}+\kappa\right)\delta_0\quad\in\mathcal{P}_{\N}$$
if $\kappa:=\sum\kappa_n\geq \sum\lambda_x$ (note that $0\leq\kappa_n\leq\kappa\leq1-\iota$, so the definition can be extended continuously to the point $\mu=\delta_0$) and to
$$\sum_{x\in Z\setminus\{pt\}}\left(\frac{2\lambda}{1-\iota}-1\right)\lambda_x\delta_x+\left(1-\frac{2\lambda^2}{1-\iota}+\lambda\right)\delta_0\quad\in\mathcal{P}_Z$$
if $\lambda:=\sum\lambda_x\geq \sum\kappa_n$.
It is easy to see that this construction is  continuous, fixes the subcomplex $\mathcal{P}_Z\cup_{pt\sim 0}\mathcal{P}_{\N}$ and respects the filtration of these locally compact spaces.
Thus, it defines a retraction 
$$\mathcal{P}_{Z\cup_{pt\sim 0}\N}\to\mathcal{P}_Z\cup_{pt\sim 0}\mathcal{P}_{\N}$$
 of $\sigma$-locally compact spaces. Similarly we see that the composition 
$$\mathcal{P}_{Z\cup_{pt\sim 0}\N}\to\mathcal{P}_Z\cup_{pt\sim 0}\mathcal{P}_{\N}\subset\mathcal{P}_{Z\cup_{pt\sim 0}\N}$$
is homotopic to the identity under linear homotopy.

Consider the short exact sequence of $\sigma$-$C^*$-algebras
$$0\to C_0(\mathcal{P}_Z\setminus\{pt\};D)\to C_0(\mathcal{P}_Z\cup_{pt\sim 0}\mathcal{P}_{\N};D)\to C_0(\mathcal{P}_{\N};D)\to0.$$
By \cite[Theorem 27]{EmeMey} we have $K_*(C_0(\mathcal{P}_{\N};D))\cong K_*(C_0(\R^+;D))=0$, so the long exact sequence in $K$-theory proves 
$$K_*(C_0(\mathcal{P}_Z\setminus\{pt\};D))\cong K_*(C_0(\mathcal{P}_Z\cup_{pt\sim 0}\mathcal{P}_{\N};D)).$$
We obtain  the second canonical isomorphism by exploiting homotopy invariance of $K$-theory:
\begin{align}
K_*(C_0(\mathcal{P}_Z\setminus\{pt\};D))&\cong K_*(C_0(\mathcal{P}_Z\cup_{pt\sim 0}\mathcal{P}_{\N};D))\nonumber
\\&\cong K_*(C_0(\mathcal{P}_{Z\cup_{pt\sim 0}\N};D))\nonumber
\\&= KX^{-*}(X^\to;D)\label{unreductiontopologicalside}
\end{align}

Thus, it remains to show that the canonical isomorphisms \eqref{unreductionanalyticside} and \eqref{unreductiontopologicalside}, coarse co-assembly and the connecting homomorphism mentioned in the statement of this lemma make the diagram
$$\xymatrix{
K_{1-*}(\cfrak(X;D))\ar[r]\ar[d]_{\cong}&K_{-*}(C_0(\mathcal{P}_Z\setminus\{pt\};D))\ar[d]^{\cong}
\\\tilde K_{1-*}(\cfrak(X^\to;D))\ar[r]&KX^*(X^\to;D)
}$$
commute. This follows easily from naturality of the connecting homomorphism in $K$-theory  applied to the diagram with exact rows
$$\xymatrix@C=1pc{
0\ar[r]&C_0(\mathcal{P}_Z\setminus\{pt\};D\otimes\K)\ar[r]\ar[d]&\cbar_0(\mathcal{P}_Z;D)\ar[r]\ar[d]&\cfrak(\mathcal{P}_Z;D)\ar[r]\ar[d]&0
\\0\ar[r]&C_0(\mathcal{P}_Z\underset{pt\sim 0}{\cup}\mathcal{P}_{\N};D\otimes\K)\ar[r]&\cbar(\mathcal{P}_Z\underset{pt\sim 0}{\cup}\mathcal{P}_{\N};D)\ar[r]&\cfrak(\mathcal{P}_Z\underset{pt\sim 0}{\cup}\mathcal{P}_{\N};D)\ar[r]&0
\\0\ar[r]&C_0(\mathcal{P}_{Z\cup_{pt\sim 0}\N};D\otimes\K)\ar[r]\ar[u]^{\simeq}&\cbar(\mathcal{P}_{Z\cup_{pt\sim 0}\N};D)\ar[r]\ar[u]&\cfrak(\mathcal{P}_{Z\cup_{pt\sim 0}\N};D)\ar[r]\ar[u]_{\cong}&0
}$$
and the obvious commutative diagram
$$\xymatrix{
K_{1-*}(\cfrak(X;D))\ar[r]^{\cong}\ar[dd]&K_{1-*}(\cfrak(\mathcal{P}_Z;D))\ar[d]
\\&K_{1-*}(\cfrak(\mathcal{P}_Z\cup_{pt\sim 0}\mathcal{P}_{\N};D))
\\K_{1-*}(\cfrak(X^\to;D))\ar[r]^{\cong}&K_{1-*}(\cfrak(\mathcal{P}_{Z\cup_{pt\sim 0}\N};D))\ar[u]_{\cong}.
}$$
\end{proof}

This Lemma justifies the following definitions of unreduced co-assembly. It will be convenient to have also unstable and uncoarsened versions at hand. For any $\sigma$-coarse space $\mathcal{X}=\bigcup_{n\in\N}X_n$  and any $C^*$-algebra $D$   denote
$$K^*(\mathcal{X};D):=K_{-*}(C_0(\mathcal{X})\otimes D)\quad \text{and}\quad K^*(\mathcal{X}):=K_{-*}(C_0(\mathcal{X})).$$
\begin{mdef}\label{coassemblydefUnredUncUnst}
Let $\mathcal{X}=\bigcup_{n\in\N}X_n$ be a $\sigma$-coarse space, $pt\in X_0$ and $D$ a $C^*$-algebra. 
The \emph{unreduced, unstable and uncoarsened coarse co-assembly map with coefficients $D$} is the connecting homomorphism
$$K_{1-*}(\kfrak(\mathcal{X};D))\to K^*(\mathcal{X}\setminus\{pt\};D)$$
associated to the short exact sequence
$$0\to C_0(\mathcal{X}\setminus\{pt\})\otimes D\to\kbar_0(\mathcal{X};D)\to\kfrak(\mathcal{X};D)\to 0.$$
\end{mdef}
\begin{mdef}
Let $X$ be a coarse space,  $pt\in X$ and  $Z\subset X$ be a countably generated, discrete coarse subspace that is coarsely equivalent to $X$ with $pt\in Z$.  We define
$${KX}^*(X\setminus\{pt\};D):=K^*(\mathcal{P}_Z\setminus\{pt\};D)=K_{-*}(C_0(\mathcal{P}_Z\setminus\{pt\})\otimes D).$$
\end{mdef}
It is known that $KX^*(X;D)=K_*(C_0(\mathcal{P}_Z)\otimes D)$ is independent of the choice of $Z\subset X$ \cite[Corollary 21]{EmeMey}. The exact sequence in $K$-theory associated to the short exact sequence
$$0\to C_0(\mathcal{P}_Z\setminus\{pt\})\otimes D\to C_0(\mathcal{P}_Z)\otimes D\to D\to 0$$
and the five lemma prove, that ${KX}^*(X\setminus\{pt\};D)$ is also independent of the choice of $Z$.

\begin{mdef}\label{coassemblydefUnredCoar}
Let $X$ be a coarse space,  $pt\in X$ and $D$ a $C^*$-algebra. The \emph{unreduced, unstable coarse co-assembly map with coefficients $D$} is  obtained by appying the uncoarsened version to the Rips complex:
$$K_{1-*}(\kfrak(X;D))\cong K_{1-*}(\kfrak(\mathcal{P}_Z;D))\to {KX}^*(X\setminus\{pt\};D)$$
Replacing $D$ by $D\otimes \K$ or even by $\K$, we obtain the \emph{unreduced, stable coarse co-assembly map (with coefficients $D$)}
\begin{align*}
K_{1-*}(\cfrak(X;D))&\to {KX}^*(X\setminus\{pt\};D)\,,
\\K_{1-*}(\cfrak(X))&\to {KX}^*(X\setminus\{pt\})\,.
\end{align*}
\end{mdef}


\section{{$K$-theory product for Higson coronas}}\label{sec:coronaringstructure}
Let $\mathcal{X}=\bigcup_{n\in \N}X_n$ be a $\sigma$-coarse space and $D,E$ coefficient $\sigma$-$C^*$-algebras. 
By multiplication of functions\footnote{The correct way of constructing this homomorphism is to use the universal property of the maximal tensor product (Proposition \ref{tensormaxunivprop}) and Corollary \ref{tensorquotients}. This method will be used implicitly a few more times, when we just prescribe the values of the map on elementary tensors. } we obtain a  homomorphism
$$\nabla:\,\kfrak(\mathcal{X};D)\otimes \kfrak(\mathcal{X};E)\to\kfrak(\mathcal{X}; D\otimes  E)\,.$$

\begin{mdef}\label{coronaproduct}
The product
$$K_i(\kfrak(\mathcal{X};D))\otimes K_j(\kfrak(\mathcal{X};E))\to K_{i+j}(\kfrak(\mathcal{X};D\otimes E))$$
is the composition of the exterior product in $K$-theory with  $\nabla_*$. Replacing $D,E$ by $D\otimes \K,E\otimes \K$ or simply both by $\K$ we obtain the products
\begin{align*}
K_i(\cfrak(\mathcal{X};D))\otimes K_j(\cfrak(\mathcal{X};E))&\to K_{i+j}(\cfrak(\mathcal{X};D\otimes E))
\\K_i(\cfrak(\mathcal{X}))\otimes K_j(\cfrak(\mathcal{X}))&\to K_{i+j}(\cfrak(\mathcal{X}))
\end{align*}
for the stable Higson corona. We denote all of them simply by ``$\,\cdot\,$''.
\end{mdef}
These product are obviously associative, graded commutative and independent of the choice of the identification $\K\otimes\K\cong \K$ which is hidden in the definition.

We are now in a position to see a first instance of the multiplicativity of the co-assembly map.
Let $Y$ be a compact metrizable space and embed it into the unit sphere of some real Hilbert space $H$. The open cone $\mathcal{O}Y$ of $Y$ is defined to be the union of all rays in $H$ starting at the origin and passing through points of $Y$. It is equipped with the subspace metric. The induced coarse structure of $\mathcal{O}Y$ is independent of the chosen embedding.
Topologically, $\mathcal{O}Y\approx Y\times[0,\infty)/Y\times\{0\}$. Furthermore, let 
$\mathcal{C}Y:= Y\times[0,\infty]/Y\times\{0\}\supset \mathcal{O}Y$ be the closed cone and $pt$ be the apex of these two cones. 

If $Y$ is a ``nice'' compact space, e.\,g.\ a compact manifold or a finite simplicial complex, then $X:=\mathcal{O}Y$ is a scalable and uniformly contractible metric space of bounded geometry. Therefore, the coarse co-assembly map of $X$ with coefficients in any $C^*$-algebra $D$ is an isomorphism
$$K_*(\cfrak(X,D))\xrightarrow[\cong]{\mu^*} K_{*-1}(C_0(X\setminus\{pt\})\otimes D)$$
by \cite[Corollary 58]{EmeMey}.

Note that continuous functions on $\bar X:=\mathcal{C}Y$ restrict to bounded continuous functions of vanishing variation on $X$. Thus, given a $C^*$-algebra $D$, there is a commutative diagram with exact rows
$$\xymatrix@C=1pc{
0\ar[r]&C_0(X\setminus\{pt\})\otimes D\ar[r]\ar[d]&C_0(\bar X\setminus\{pt\})\otimes D\ar[r]\ar[d]&C(Y)\otimes D\ar[r]\ar[d]_{p^*}&0
\\0\ar[r]&C_0(X\setminus\{pt\})\otimes D\otimes\K\ar[r]&\cbar_0(X,D)\ar[r]&\cfrak(X,D)\ar[r]&0
}$$
with vertical maps defined by choosing a rank one projection in $\K$. The induced right vertical arrow $p^*$ is in fact given by pulling back functions along the projection $X\setminus\{pt\}\to Y$.
We obtain a commutative diagram
$$\xymatrix{
K_*(C(Y)\otimes D)\ar[r]^{\partial}_\cong\ar[d]_{p^*}&K_{*-1}(C_0(X\setminus\{pt\}))\ar@{=}[d]
\\K_*(\cfrak(X,D))\ar[r]^{\mu^*}_\cong&K_{*-1}(C_0(X\setminus\{pt\})).
}$$
After identifying $K_{*-1}(C_0(X\setminus\{pt\}))$ with $K_*(C(Y)\otimes D)$ via the isomorphism $\partial$ we see that $p^*$ is inverse to the coarse co-assembly map $\tilde\mu^*:=\partial^{-1}\circ\mu^*$.

It is now a triviality to check that $p^*$ and therefore also $\tilde\mu^*$ are multiplicative in the following sense:
\begin{prop}\label{prop:openconecoassembly}
Let $X:=\mathcal{O}Y$ be the open cone over a  compact metrizable space $Y$ and assume that the topology  of $Y$ is ``nice'' enough such that $X$ is  scalable and uniformly contractible of bounded geometry. If $D,E$ are any $C^*$-algebras, then the diagram
$$\xymatrix{
K_i(\cfrak(X,D))\otimes K_j(\cfrak(X,E))\ar[r]\ar[d]^{\tilde\mu^*\otimes\tilde\mu^*}&K_{i+j}(\cfrak(X,D\otimes E))\ar[d]^{\tilde\mu^*}
\\K_i(C(Y)\otimes D)\otimes K_j(C(Y)\otimes E)\ar[r]&K_{i+j}(C(Y)\otimes D\otimes E)
}$$
commutes. In particular, the co-assembly map 
$K_*(\cfrak(X))\to K^{-*}(Y)$
is a ring isomorphism.
\end{prop}

For more general $\sigma$-coarse spaces $\mathcal{X}$ we have to construct a secondary product directly on the groups $K_*(C_0(\mathcal{X}\setminus\{pt\})\otimes D)$. This is done in the next section. The key property of cones which we will have to continue to assume is contractibility to the point $pt$.


\section{{Secondary product for $\sigma$-contractible spaces}}\label{sec:secondaryproduct}
Let $\mathcal{X}=\bigcup_{n\in\N}X_n$ 
 be a $\sigma$-locally compact space and $D$ a $\sigma$-$C^*$-algebra. Recall the definitions
$$K^*(\mathcal{X};D):=K_{-*}(C_0(\mathcal{X})\otimes D)\quad \text{and}\quad K^*(\mathcal{X}):=K_{-*}(C_0(\mathcal{X})).$$
These groups are contravariantly functorial under proper continuous maps between $\sigma$-locally compact spaces and covariantly functorial under $*$-ho\-mo\-mor\-phisms between coefficient $\sigma$-$C^*$-algebras.

The exterior product in $K$-theory of $\sigma$-$C^*$-algebras specializes to an exterior product 
$$K^i(\mathcal{X};D)\otimes K^j(\mathcal{Y};E)\to K^{i+j}(\mathcal{X}\times\mathcal{Y};D\otimes E).$$

In this section, we will construct the secondary product
$$K^i(\mathcal{X}\setminus\{pt\};D)\otimes K^j(\mathcal{X}\setminus\{pt\};E)\to K^{i+j-1}(\mathcal{X}\setminus\{pt\};D\otimes E)$$
under the assumption that $\mathcal{X}$ is equipped with a $\sigma$-contraction onto the point $pt$ as defined below,
and prove the most basic properties like associativity.

\begin{mdef}
Let $\mathcal{X}=\bigcup_{n\in\N}X_n$ be a $\sigma$-locally compact space. We call a map
$$H:\mathcal{X}\times [0,1]\to\mathcal{X}$$
a \emph{$\sigma$-contraction} to the point $pt$ iff
\begin{itemize}\itemsep0pt\parskip0pt 
\item $H|_{\mathcal{X}\times\{0\}}=\id_{\mathcal{X}}$ and $H(\mathcal{X}\times\{1\}\cup\{pt\}\times [0,1])=\{pt\}$,
\item for each $n$ there is $m\geq n$ such that $H(X_n\times [0,1])\subset X_m$ and
the restriction $$H_{X_n\times [0,1]}:\,X_n\times [0,1]\to X_m$$
is continuous. 
\end{itemize}
\end{mdef}
For a locally compact space $X$ (i.\,e.\ $X_n=X\,\forall n$), a $\sigma$-contraction is nothing but a contraction which does not move the point $pt$.

Given such a $\sigma$-contraction on $\mathcal{X}$, we can define the map
\begin{align*}
\Gamma:\,[0,1]\times \mathcal{X}&\to \mathcal{X}\times \mathcal{X}
\\(t,x)&\mapsto\begin{cases}(H(x,1-2t),x)&t\leq1/2\\(x,H(x,2t-1))&t\geq1/2\end{cases}
\end{align*}

Note that $\Gamma(\,[0,1]\times\{pt\}\cup\{0,1\}\times \mathcal{X}\,)\subset\mathcal{X}\times\{pt\}\cup\{pt\}\times\mathcal{X}$ and one entry of the pair $\Gamma(x,t)$ is always $x$.
These conditions ensure that it restricts to a proper continuous map (a morphism, cf.\ Definition \ref{sigmaspacesmonoidalcategory})
$$\Gamma:\,(0,1)\times (\mathcal{X}\setminus\{pt\})\to (\mathcal{X}\setminus\{pt\})\times (\mathcal{X}\setminus\{pt\}).$$
(To be precise: the map defining this morphism is defined on the open subset $\mathcal{U}=\Gamma^{-1}((\mathcal{X}\setminus\{pt\})\times (\mathcal{X}\setminus\{pt\}))$.)

\begin{mdef}\label{secondaryproductdefinition}
Let $\mathcal{X}=\bigcup_{n\in\N}X_n$ be a $\sigma$-locally compact space together with a $\sigma$-contraction $H$ and $D,E$ coefficient $C^*$-algebras. The \emph{secondary product} is $(-1)^i$ times the composition
\begin{align*}
K^i(\mathcal{X}\setminus\{pt\};D)&\otimes K^j(\mathcal{X}\setminus\{pt\};E)\to
\\&\xrightarrow[\hspace{2.5ex}]{} K^{i+j}(\,(\mathcal{X}\setminus\{pt\})\times(\mathcal{X}\setminus\{pt\});\,D\otimes E)
\\&\xrightarrow[\hspace{2.5ex}]{\Gamma^*} K^{i+j}(\,(0,1)\times (\mathcal{X}\setminus\{pt\});\,D\otimes E)
\\&\xrightarrow[\hspace{2.5ex}]{\cong} K^{i+j-1}(\mathcal{X}\setminus\{pt\};D\otimes E),
\end{align*}
where the last isomorphism is the inverse of the connecting homomorphism associated to the short exact sequence
\begin{align*}
0\to C_0(0,1)\to C_0[0,1)\to\C\to 0
\end{align*}
tensored with $C_0(\mathcal{X}\setminus\{pt\})\otimes D\otimes E$.
\end{mdef}
We shall use the infix notation $x\otimes y\mapsto x*y$ for this secondary multiplication.
The remaining part of this section is devoted to proving basic properties of the secondary product.
\begin{prop}
The secondary product is independent of the choice of the $\sigma$-contraction to the given point $pt$ and graded commutative.
\end{prop}
\begin{proof}
If $\tilde H$ is another $\sigma$-contraction to the same point $pt$ with associated proper continuous map $\tilde\Gamma$, then both $\Gamma$ and $\tilde\Gamma$ are homotopic to the proper continuous map
$$(t,x)\mapsto\begin{cases}(H(\tilde H(x,1-2t),1-2t),\;x)&t\leq1/2\\(x,\;H(\tilde H(x,2t-1),2t-1))&t\geq1/2\end{cases}$$
in the obvious way and thus induce the same map in $K$-theory.

Exchanging the factors in the exterior product gives an additional sign $(-1)^{ij}$, changing the orientation of the interval $(0,1)$ gives another $-1$ and instead of multiplying with $(-1)^i$ we have to multiply with $(-1)^j$. In total we obtain a change in sign by $(-1)^{(i+1)(j+1)}$. Note that this is the desired prefactor, as the correct degree of $K^i(\mathcal{X}\setminus\{pt\};D)$ is in fact $i-1$, not $i$.
\end{proof}

\begin{thm}\label{thm:secprodassoc}
The secondary product is associative.
\end{thm}
\begin{proof}
For space reasons, we prove this for all coefficient algebras set to $\C$. The proof for the general case is obtained by simply dropping in the coefficient $C^*$-algebras $D,E,F$ at the appropriate places.

Consider the diagram in Figure \ref{fig:secprodassoc}.
\begin{sidewaysfigure}
$$\xymatrix{
{\begin{matrix}K^i(\mathcal{X}\setminus\{pt\})\otimes\\\otimes K^j(\mathcal{X}\setminus\{pt\})\\\otimes K^k(\mathcal{X}\setminus\{pt\})\end{matrix}}\ar[r]\ar[d]&
{\begin{matrix}K^{i+j}(\,(\mathcal{X}\setminus\{pt\})\times\\\qquad\times(\mathcal{X}\setminus\{pt\})\,)\\\otimes K^k(\mathcal{X}\setminus\{pt\})\end{matrix}}\ar[r]^{\Gamma^*\otimes\id}\ar[d]&
{\begin{matrix}K^{i+j}(\,(0,1)\times\\\qquad\times(\mathcal{X}\setminus\{pt\})\,)\\\otimes K^k(\mathcal{X}\setminus\{pt\})\end{matrix}}\ar[d]&
{\begin{matrix}K^{i+j-1}(\mathcal{X}\setminus\{pt\})\\\otimes K^k(\mathcal{X}\setminus\{pt\})\end{matrix}}\ar[d]\ar[l]_{\delta_1\otimes\id}^{\cong}
\\{\begin{matrix} K^i(\mathcal{X}\setminus\{pt\})\otimes\\ K^{j+k}(\,(\mathcal{X}\setminus\{pt\})\times\\\qquad\times(\mathcal{X}\setminus\{pt\})\,)\end{matrix}}\ar[r]\ar[d]^{\id\otimes\Gamma^*}&
{\begin{matrix}K^{i+j+k}(\,(\mathcal{X}\setminus\{pt\})\\\qquad\times(\mathcal{X}\setminus\{pt\})\\\qquad\times(\mathcal{X}\setminus\{pt\})\,)\end{matrix}}\ar[r]^{(\Gamma\times\id)^*}\ar[d]_{(\id\times\Gamma)^*}&
{\begin{matrix}K^{i+j+k}(\,(0,1)\\\qquad\times(\mathcal{X}\setminus\{pt\})\\\qquad\times(\mathcal{X}\setminus\{pt\})\,)\end{matrix}}\ar[d]^{(\id\times \Gamma)^*}&
{\begin{matrix}K^{i+j+k-1}((\mathcal{X}\setminus\{pt\})\\\qquad\times(\mathcal{X}\setminus\{pt\})\,)\end{matrix}}\ar[l]_{\delta_2}^{\cong}\ar[d]^{\Gamma^*}
\\{\begin{matrix} K^i(\mathcal{X}\setminus\{pt\})\otimes\\ K^{j+k}(\,(0,1)\times\\\qquad\times(\mathcal{X}\setminus\{pt\})\,)\end{matrix}}\ar[r]&
{\begin{matrix}K^{i+j+k}(\,(\mathcal{X}\setminus\{pt\})\\\qquad\times(0,1)\\\qquad\times(\mathcal{X}\setminus\{pt\})\,)\end{matrix}}\ar[r]_{\tilde\Gamma^*}&
{\begin{matrix}K^{i+j+k}(\,(0,1)^2\\\quad\times(\mathcal{X}\setminus\{pt\})\,)\end{matrix}}&
{\begin{matrix}K^{i+j+k-1}((0,1)\\\qquad\times(\mathcal{X}\setminus\{pt\})\,)\end{matrix}}\ar[l]_{\cong}^{\delta_3}
\\{\begin{matrix} K^i(\mathcal{X}\setminus\{pt\})\otimes\\ K^{j+k-1}(\mathcal{X}\setminus\{pt\})\end{matrix}}\ar[r]\ar[u]^{\cong}_{\id\otimes\delta_1}&
{\begin{matrix}K^{i+j+k-1}((\mathcal{X}\setminus\{pt\})\\\qquad\times(\mathcal{X}\setminus\{pt\})\,)\end{matrix}}\ar[r]^{\Gamma^*}\ar[u]_{\delta_5}^{\cong}&
{\begin{matrix}K^{i+j+k-1}((0,1)\\\qquad\times(\mathcal{X}\setminus\{pt\})\,)\end{matrix}}\ar[u]_{\delta_4}^{\cong}\ar@{=}[ur]&
K^{i+j+k-2}(\mathcal{X}\setminus\{pt\})\ar[l]_{\cong}^{\delta_1}\ar[u]_{\delta_1}^{\cong}
}$$
\caption{Proving Theorem \ref{thm:secprodassoc}}\label{fig:secprodassoc}
\end{sidewaysfigure}
The upper vertical and left horizontal maps are exterior multiplication. All the $\delta_i$ are connecting homomorphisms associated to inclusions of closed subsets: 
\begin{description}
\item[ $\delta_1$] is associated to the inclusion $\mathcal{X}\setminus\{pt\}\to [0,1)\times(\mathcal{X}\setminus\{pt\}),\;x\mapsto (0,x)$.\item[ $\delta_2$] is associated to the inclusion 
\begin{align*}
(\mathcal{X}\setminus\{pt\})\times(\mathcal{X}\setminus\{pt\})&\to [0,1)\times(\mathcal{X}\setminus\{pt\})\times(\mathcal{X}\setminus\{pt\}),
\\(x,y)&\mapsto (0,x,y).
\end{align*}
\item[ $\delta_3$] is associated to the inclusion 
\begin{align*}
(0,1)\times(\mathcal{X}\setminus\{pt\})&\to  [0,1)\times(0,1)\times(\mathcal{X}\setminus\{pt\})
\\(s,x)&\mapsto (0,s,x).
\end{align*}
\item[ $\delta_4$] is associated to the inclusion 
\begin{align*}
(0,1)\times(\mathcal{X}\setminus\{pt\})&\to  (0,1]\times(0,1)\times(\mathcal{X}\setminus\{pt\})
\\(s,x)&\mapsto (1,s,x).
\end{align*}
\item[ $\delta_5$] is associated to the inclusion 
\begin{align*}
(\mathcal{X}\setminus\{pt\})\times(\mathcal{X}\setminus\{pt\})&\to (\mathcal{X}\setminus\{pt\})\times[0,1)\times(\mathcal{X}\setminus\{pt\})
\\(x,y)&\mapsto (x,0,y).
\end{align*}
\end{description}
Note that all of them are isomorphisms, because $C_0[0,1)$ and $C_0(0,1]$ are contractible. 
 Finally, 
\begin{align*}
\tilde\Gamma:\,(0,1)^2\times\mathcal{X}\setminus\{pt\}&\to\mathcal{X}\setminus\{pt\}\times(0,1)\times\mathcal{X}\setminus\{pt\}
\\(s,t,x)&\mapsto\begin{cases}(\,H(x,1-2t),\,1-s,\,x\,)&t\leq1/2\\(\,x,\,1-s,\,H(x,2t-1)\,)&t\geq1/2.\end{cases}
\end{align*}
The upper left square commutes by associativity of the exterior product.
The middle squares on the upper and left side commute by naturality of the exterior product.
The the lower left square commutes up to a sign $(-1)^i$ and the upper right square commutes by Axiom \ref{Ktheoryaxiomcompatibility}.
In the lower right square, $\delta_4$ is the negative of $\delta_3$ by Axiom \ref{Ktheoryaxiomreflection}.

The middle squares on the bottom and right side commute by naturality of the connecting homomorphism under the proper continuous maps of pairs of spaces
\begin{align*}
\big([0,1)\times&(0,1)\times(\mathcal{X}\setminus\{pt\}),\,\{0\}\times(0,1)\times(\mathcal{X}\setminus\{pt\})\big)
\\&\to\big([0,1)\times(\mathcal{X}\setminus\{pt\})\times(\mathcal{X}\setminus\{pt\}),\,\{0\}\times(\mathcal{X}\setminus\{pt\})\times(\mathcal{X}\setminus\{pt\})\big)
\\(s,t,x)&\mapsto (s,\Gamma(t,x))
\end{align*}
and
\begin{align*}
\big((0,1]\times&(0,1)\times(\mathcal{X}\setminus\{pt\}),\,\{1\}\times(0,1)\times(\mathcal{X}\setminus\{pt\})\big)
\\&\to\big((\mathcal{X}\setminus\{pt\})\times[0,1)\times(\mathcal{X}\setminus\{pt\}),\,(\mathcal{X}\setminus\{pt\})\times\{0\}\times(\mathcal{X}\setminus\{pt\})\big)
\\(s,t,x)&\mapsto\begin{cases}(\,H(x,1-2t),\,1-s,\,x\,)&t\leq1/2\\(\,x,\,1-s,\,H(x,2t-1)\,)&t\geq1/2.\end{cases}
\end{align*}

It remains to prove commutativity of the middle square. The  proper continuous maps $(0,1)^2\times\mathcal{X}\setminus\{pt\}\to (\mathcal{X}\setminus\{pt\})^3$ inducing the two compositions are
\begin{align*}
(\Gamma\times\id)&\circ(\id\times\Gamma)(s,t,x)=\begin{cases}(\Gamma(s,H(x,1-2t)),\;x)&t\leq1/2\\(\Gamma(s,x),\;H(x,2t-1))&t\geq1/2\end{cases}
\\&=\begin{cases}
(\,H(H(x,1-2t),1-2s),\,H(x,1-2t),\,x\,)&s\leq1/2,t\leq1/2
\\(\,H(x,1-2s),\,x,\,H(x,2t-1)\,)&s\leq1/2,t\geq1/2
\\(\,H(x,1-2t),\,H(H(x,1-2t),2s-1),\,x\,)&s\geq1/2,t\leq1/2
\\(\,x,\,H(x,2s-1),\,H(x,2t-1)\,)&s\geq1/2,t\geq1/2
\end{cases}
\\(\id\times\Gamma)&\circ\tilde\Gamma(s,t,x)=\begin{cases}(\,H(x,1-2t),\,\Gamma(1-s,\,x)\,)&t\leq1/2\\(\,x,\,\Gamma(1-s,\,H(x,2t-1))\,)&t\geq1/2\end{cases}
\\&=\begin{cases}
(\,H(x,1-2t),\,x,\,H(x,1-2s\,)&s\leq1/2,t\leq1/2
\\(\,x,\,H(x,2t-1),\,H(H(x,2t-1),1-2s)\,)&s\leq1/2,t\geq1/2
\\(\,H(x,1-2t),\,H(x,2s-1),x\,)&s\geq1/2,t\leq1/2
\\(\,x,\,H(H(x,2t-1),2s-1),\,H(x,2t-1)\,)&s\geq1/2,t\geq1/2.
\end{cases}
\end{align*}
These two maps are homotopic, as can be seen by performing the following four consecutive homotopies:
\begin{enumerate}
\item Homotop within $t\leq1/2$ (always leaving the rest constant):
$$(r,s,t,x)\mapsto\begin{cases}
(H(H(x,1-2t),1-2s),H(x,(1-r)(1-2t)),x)&s\leq 1/2
\\(H(x,1-2t),H(H(x,(1-r)(1-2t)),2s-1),x)&s\geq 1/2
\end{cases}$$
\item Homotop within $s\leq 1/2$:
$$(r,s,t,x)\mapsto\begin{cases}
(H(H(x,1-2t),1-2s),x,H(x,r(1-2s)))&t\leq 1/2
\\(H(x,1-2s),x,H(H(x,2t-1),r(1-2s)))&t\geq 1/2
\end{cases}$$
\item Homotop again within $s\leq 1/2$:
$$(r,s,t,x)\mapsto\begin{cases}
(H(H(x,1-2t),(1-r)(1-2s)),x,H(x,1-2s))&t\leq 1/2
\\(H(x,(1-r)(1-2s)),x,H(H(x,2t-1),1-2s))&t\geq 1/2
\end{cases}$$
\item Homotop within $t\geq 1/2$:
$$(r,s,t,x)\mapsto\begin{cases}
(x,H(x,r(2t-1)),H(H(x,2t-1),1-2s))&s\leq 1/2
\\(x,H(H(x,r(2t-1)),2s-1),H(x,2t-1))&s\geq 1/2
\end{cases}$$\end{enumerate}
Checking continuity is straightforward. To see properness, note that for any $(r,s,t,x)$ there is always one component equal to $x$, and if $s$ or $t$ is set equal to $0$ or $1$, then one component is equal to $pt$.

Composing the maps on the left and bottom side yields
$$x\otimes y\otimes z\mapsto (-1)^jx\otimes (y*z)\mapsto (-1)^{i+j}x*(y*z)$$
while the composition along top and right side is 
$$x\otimes y\otimes z\mapsto (-1)^i(x*y)\otimes z\mapsto (-1)^{i+(i+j-1)}(x*y)*z\,.$$
As we have seen above, the diagram commutes up to a total sign of $(-1)^{i-1}$. This implies $x*(y*z)=(x*y)*z$.
\end{proof}


\section{Multiplicativity of the coarse co-assembly map}\label{sec:multiplicativity}
\begin{thm}\label{maintheorem}
Let $\mathcal{X}=\bigcup_{n\in \N}X_n$ be a $\sigma$-coarse space together with a $\sigma$-contraction $H$ and let $D,E$ be coefficient $\sigma$-$C^*$-algebras. Then the products we have constructed and  coarse co-assembly fit into a commutative diagram.
\begin{equation}\xymatrix{
K_{i}(\kfrak(\mathcal{X};D))\otimes K_{j}(\kfrak(\mathcal{X};E))\ar[r]
\ar[d]&K^{1-i}(\mathcal{X}\setminus\{pt\};D)\otimes K^{1-j}(\mathcal{X}\setminus\{pt\};E)\ar[d]
\\K_{i+j}(\kfrak(\mathcal{X};D\otimes E))\ar[r]
&K^{1-i-j}(\mathcal{X}\setminus\{pt\};D\otimes E)
}\end{equation}
\end{thm}
\begin{proof}
We will use the following abbreviations: Let
$$I=C_0(\mathcal{X}\setminus\{pt\}),\quad I_D=C_0(\mathcal{X}\setminus\{pt\};D)=I\otimes D,\quad A_D=\kbar_0(\mathcal{X};D),$$
so that $$A_D/I_D=\kfrak(\mathcal{X};D),$$
and $I_E,A_E,I_{D\otimes E},A_{D\otimes E}$ are defined analogously. Furthermore, let
\begin{align*}
A_1&=C_b(\mathcal{X},pt;\,D\otimes E),
\\A_2&=C_b(\,[0,1]\times\mathcal{X},\,[0,1]\times\{pt\}\cup\{0,1\}\times\mathcal{X};\,D\otimes E),
\\A_3&=C_b(\,[0,1]\times\mathcal{X},\,[0,1]\times\{pt\}\cup\{0\}\times\mathcal{X};\,D\otimes E),
\\I_2&=C_0(\,(0,1)\times (\mathcal{X}\setminus\{pt\});\,D\otimes E)
\\&=C_0(0,1)\otimes I\otimes D\otimes E\quad=C_0(0,1)\otimes I_{D\otimes E},
\\I_3&=C_0(\,(0,1]\times(\mathcal{X}\setminus\{pt\});\,D\otimes E)
\\&=C_0(0,1]\otimes I\otimes D\otimes E\quad=C_0(0,1]\otimes I_{D\otimes E},
\\I_4&=C_0(\,[0,1)\times(\mathcal{X}\setminus\{pt\});\,D\otimes E)
\\&=C_0[0,1)\otimes I\otimes D\otimes E\quad=C_0[0,1)\otimes I_{D\otimes E}.
\end{align*}
We will make use of the naturality of the connecting homomorphism in $K$-theory with respect to the following 
commutative diagrams with exact rows:
\begin{enumerate}
\item There is a homomorphism 
\begin{align*}
\alpha:\,I_D\otimes A_E&\to A_2
\\f\otimes g&\mapsto\,\left[(t,x)\mapsto\begin{cases}f(H(x,1-2t))\otimes g(x)&t\leq 1/2\\f(x)\otimes  g(H(x,2t-1))&t\geq 1/2\end{cases}\right]\,.
\end{align*}
It maps $I_D\otimes I_E$ to $I_2=C_0(0,1)\otimes I_{D\otimes E}$ and the restriction of $\alpha$ to $I_D\otimes I_E$ is in fact $\Gamma^*$.
Thus we obtain the commutative diagram with exact rows
$$\xymatrix{
0\ar[r]&I_D\otimes I_E\ar[r]\ar[d]^{\Gamma^*}&I_D\otimes  A_E\ar[r]\ar[d]^{\alpha}& I_D\otimes  A_E/I_E\ar[r]\ar[d]^{\bar\alpha}&0
\\0\ar[r]&I_2\ar[r]&A_2\ar[r]&A_2/I_2\ar[r]&0
}$$
with  $\bar\alpha$ being the quotient of $\alpha$.

\item The quotient homomorphism $\bar\alpha$ from the previous diagram also appears  in the diagram
$$\xymatrix{
0\ar[r]&I_D\otimes A_E/I_E\ar[r]\ar[d]^{\bar\alpha}&A_D\otimes  A_E/I_E\ar[r]\ar[d]^{\beta}& A_D/I_D\otimes  A_E/I_E\ar[r]\ar[d]^{\bar\beta}&0
\\0\ar[r]&A_2/I_2\ar[r]&A_3/I_3\ar[r]&A_1/I_{D\otimes E}\ar[r]&0.
}$$
The lower row consists of the canonical maps: quotients of the inclusion $A_2\subset A_3$ and  evaluation at one $A_3\to A_1$. It is easily seen to be an exact sequence.

The middle vertical homomorphism $\beta:A_D\otimes A_E/I_E\to A_3/I_3$ is defined  by mapping $f\otimes\bar g$ to the residue class of
$$(t,x)\mapsto\begin{cases}f(H(x,1-2t))\otimes g(x)&t\leq 1/2\\f(x)\otimes g(x)&t\geq 1/2\,.\end{cases}$$

We quickly check commutativity of the left square in the diagram by considering the images of elementary tensors $f\otimes\bar g\in I_D\otimes A_E/I_E$. The lower left path takes it to the residue class of
$$(t,x)\mapsto\begin{cases}f(H(x,1-2t))\otimes g(x)&t\leq 1/2\\f(x)\otimes g(H(x,2t-1))&t\geq 1/2\,.\end{cases}$$
Its difference to the representative obtained by following the upper right path is
$$(t,x)\mapsto\begin{cases}0&t\leq 1/2\\f(x)\otimes [g(H(x,2t-1))-g(x)]&t\geq 1/2\,,\end{cases}$$
which
is contained in $I_3$, because $f$ vanishes at infinity.
Thus, the two compositions agree on elementary tensors and consequently on all elements.

By exactness of both rows and commutativity of the left square, $\beta$ passes to the quotient and we obtain
$$\bar\beta:\,A_D/I_D\otimes A_E/I_E\to A_1/I_{D\otimes E}.$$

\item The homomorphism
$$\gamma:\,I_4\to A_2,\,f\mapsto\left[(t,x)\mapsto\begin{cases}f(0,H(x,1-2t))&t\leq1/2\\f(2t-1,x)&t\geq1/2\end{cases}\right]$$
restricts to a the homomorphism 
$$\gamma':\,I_2\to I_2,\,f\mapsto\left[(t,x)\mapsto\begin{cases}0&t\leq1/2\\f(2t-1,x)&t\geq1/2\end{cases}\right]$$
which is obviously homotopic to the identity. They induce the third diagram of exact sequences:
$$\xymatrix{
0\ar[r]&I_2\ar[r]&A_2\ar[r]&A_2/I_2\ar[r]&0
\\0\ar[r]&I_2\ar[r]\ar[u]^{\gamma'}&I_4\ar[r]\ar[u]^{\gamma}&I_{D\otimes E}\ar[r]\ar[u]^{\bar\gamma}&0
}$$
The induced homomorphism $\bar\gamma$ is easily seen to map $f\in I_{D\otimes E}$ to the residue class of
\begin{equation}\label{formulagamma3}
(t,x)\mapsto\begin{cases}f(H(x,1-2t))&t\leq1/2\\2(1-t)\cdot f(x)&t\geq1/2.\end{cases}
\end{equation}

Note furthermore, that the connecting homomorphism
$$K_*(I_{D\otimes E})\xrightarrow{\cong}K_{*+1}(I_2)=K_{*+1}(C_0(0,1)\otimes I_{D\otimes E})$$
is our preferred identification of these two groups.

\item The fourth is
$$\xymatrix{
0\ar[r]&A_2/I_2\ar[r]&A_3/I_3\ar[r]&A_1/I_{D\otimes E}\ar[r]&0
\\0\ar[r]&I_{D\otimes E}\ar[r]\ar[u]^{\bar\gamma}&A_{D\otimes E}\ar[r]\ar[u]^{\delta}&A_{D\otimes E}/I_{D\otimes E}\ar[r]\ar[u]^{\bar\delta}&0
}$$
with $\bar\gamma$ as above and $\delta$ mapping $f\in A_{D\otimes E}$ to the residue class of
$$(t,x)\mapsto\begin{cases}f(H(x,1-2t))&t\leq1/2\\ f(x)&t\geq1/2.\end{cases}$$
Note that for $f\in I_{D\otimes E}$ the difference between this function and the one given by \eqref{formulagamma3} lies in $I_3$. Thus, the left square commutes and the quotient homomorphism $\bar\delta$ is simply inclusion.
\end{enumerate}
We now obtain the diagram
$$\xymatrix{
K_i\left(\frac{A_D}{I_D}\right)\otimes K_j\left(\frac{A_E}{I_E}\right)\ar[r]\ar[d]&K_{i-1}(I_D)\otimes K_{j}\left(\frac{A_E}{I_E}\right)\ar[r]\ar[d]&K_{i-1}(I_D)\otimes K_{j-1}(I_E)\ar[d]
\\K_{i+j}\left(\frac{A_D}{I_D}\otimes \frac{A_E}{I_E}\right)\ar[r]\ar[d]^{\bar\beta_*}\ar@/_3pc/[dd]_{\nabla_*}&K_{i+j-1}\left(I_D\otimes \frac{A_E}{I_E}\right)\ar[r]\ar[d]^{\bar\alpha_*}&K_{i+j-2}(I_D\otimes I_E)\ar[d]^{\Gamma^*}
\\K_{i+j}\left(\frac{A_1}{I_{D\otimes E}}\right)\ar[r]&K_{i+j-1}\left(\frac{A_2}{I_2}\right)\ar[r]&K_{i+j-2}(I_2)
\\K_{i+j}\left(\frac{A_{D\otimes E}}{I_{D\otimes E}}\right)\ar[r]\ar[u]_{\bar\delta_*}&K_{i+j-1}(I_{D\otimes E})\ar[r]^{\cong}\ar[u]_{\bar\gamma_*}&K_{i+j-2}(I_2).\ar@{=}[u]_{\gamma'_*=\id}
}$$
All horizontal maps are connecting homomorphisms and the upper vertical maps are exterior products.

The upper left square commutes and the upper right commutes up to a sign  $(-1)^{i-1}$ by Axiom \ref{Ktheoryaxiomcompatibility}.
The lower four square commute by naturality of the connecting homomorphisms (Axiom \ref{Ktheoryaxiomsequence}) under the morphisms of exact sequences  just constructed.
The bulge on the left already commutes on the level of homomorphisms: $\bar\beta=\bar\delta\circ\nabla$

The outer left path from the upper left corner to the middle of the bottom side is multiplication on $K_*(\kfrak(\mathcal{X};\_))$ followed by co-assembly and the outer right path is co-assembly applied twice followed by the secondary product on $K^*(\mathcal{X};\_)$ multiplied by $(-1)^{i-1}$. Exactly this factor appears inside the diagram, so the claim follows.
\end{proof}


\section{Coarse contractibility}\label{sec:coarsecontractability}
So far we studied multiplicativity of uncoarsened co-assembly maps. In this section we finally transfer the results to the coarsened versions. The required $\sigma$-contractibility of the Rips complex is provided by the following notion of coarse contractibility. 
\begin{mdef}\label{coarsecontraction}
Let $X$ be a coarse space. We call a Borel map
$$H:X\times\N\to X$$
a \emph{coarse contraction} to a point $pt\in X$ (and $X$ \emph{coarsely contractible}), if 
\begin{itemize}\itemsep0pt\parskip0pt 
\item $H\times H$ maps entourages of the product coarse structure on $X\times\N$ to entourages of $X$,
\item the restriction of $H$ to $X\times\{0\}$ is the identity on $X$,
\item for any bounded subset $B\subset X$ there is $N_B\in\N$ such that $H(x,n)=pt$ for all $(x,n)\in B\times \{n:n\geq N_B\}$,
\item $H(pt,n)=pt$ for all $n\in\N$.
\end{itemize}
\end{mdef}
This map lacks properness and is thus no coarse map. 
However, the induced map
$$\{(x,n)\in X\times\N\,:\,n\leq N_{\{x\}}\}\to X\times X,\quad (x,n)\mapsto(H(x,n),x)$$
is a coarse map. Note the similarity to our earlier definition of the map $\Gamma$.

Coarse contractibility is obviously invariant under coarse equivalences such as passing from a countably generated coarse space $X$ to a coarsely equivalent discrete coarse subspace $Z\subset X$.

Now here are some examples of coarse contractibility:
\begin{example}
Open cones $\mathcal{O}Y\approx Y\times[0,\infty)/Y\times\{0\}$ are coarsely contractible for any compact metrizable space $Y$. The coarse contraction $H$ is defined by $H((y,t),n):=(y,\max\{t-n,0\})$. The same argument shows that the foliated cones of \cite{RoeFoliations} are coarsely contractible.
\end{example}
\begin{example}
Locally compact, complete $\operatorname{CAT}(0)$ spaces $X$ are coarsely contractible by moving a point $y\in X$ with constant speed along the unique geodesic from $y$ to some fixed base point $p\in X$.
\end{example}
We only assume locally compactness and completeness in this and the following example, because our definition of coarse spaces involved a compatible locally compact topology on the underlying set. This premise becomes unnecessary, if we define coarse contractibility more generally for sets equipped with a coarse structure.
\begin{example}
In this example we show that locally compact, complete hyperbolic metric spaces (cf.\ \cite{BriHae,GroHyperbolic}) are coarsely contractible. 
So in particular, Gromov's hyperbolic groups equipped with an arbitrary word metric are coarsely contractible.

We shall first recall the definition of hyperbolic metric spaces as presented in \cite[Section 4.2]{HigGue}.
Let $X$ be a metric space. A \emph{geodesic segment} in $X$ is a curve $\gamma:[a,b]\to X$ such that 
$$d(\gamma(s),\gamma(t))=|s-t|\quad\forall s,t\in [a,b]\,.$$
A \emph{geodesic metric space} is a metric space in which each two points are connected by a geodesic segment.
A \emph{geodesic triangle} in $X$ consists of three points of $X$ and for each two of these points a geodesic segment connecting them.
A geodesic triangle is called $D$-slim, $D\geq0$, if each point on each edge has distance at most $D$ from one of the other two edges.
\begin{mdef}
A geodesic metric space $X$ is called \emph{$D$-hyperbolic}, if every geodesic triangle in $X$ is $D$-slim. It is called \emph{hyperbolic} if it $D$-hyperbolic for some $D\geq0$.
\end{mdef}
Now let $X$ be a $D$-hyperbolic space and $pt\in X$. For each $x\in X$ we choose a geodesic segment $\gamma_x:[0,l_x]\to X$ with $\gamma(0)=x$ and $\gamma(l_x)=pt$,  so $l_x=d(x,pt)$. We extend each $\gamma_x$ to $[0,\infty)$ by $\gamma(t)=pt$ for $t>l_x$ and define
$$H:\,X\times\N\to X,\quad (x,n)\mapsto \gamma_x(n)\,.$$
We claim that $H$ is a coarse contraction to the point $pt$.\footnote{We do not care whether this map is actually a Borel map, because we may always choose a  discrete, coarsely equivalent subspace $Z\subset X$, restrict $H$ to $Z\times\N$ and finally re-extend to $X\times\N$ in a Borel fashion to obtain a Borel coarse contraction. }
Clearly $H|_{X\times\{0\}}=\id_X$, $H(pt,n)=pt$ for all $n\in \N$ and if $B\subset X$ is bounded and $x\in B$, then $H(x,n)=pt$ for all $n\geq N_B:=\sup_{x\in B}d(pt,x)$.

It remains to show that entourages are mapped to entourages. 
Let $R>0$. We claim that for all $x,y\in X$ with $d(x,y)\leq R$ and $n\in\N$ the inequality $d(H(x,n),H(y,n))\leq 4D+R$ holds. 
Denote the geodesic triangle consisting of the points $pt,x,y$, the geodesics $\gamma_x,\gamma_y$ and another geodesic $\overline{xy}$ between $x,y$ by $\triangle$.
We may assume w.\,l.\,o.\,g.\ that $l_x\geq l_y$. The triangle inequality in $\Delta$ implies $\delta:=l_x-l_y\leq d(x,y)\leq R$.

Assume that there is $n\in\N$ such that 
$$M:=d(\gamma_x(n),\gamma_y(n))=d(H(x,n),H(x,m))>4D+R\,.$$
Define $z:=\gamma_x(n)$, $w:=\gamma_y(n)$. 
Clearly $n< l_y$, because otherwise $d(z,w)=d(z,pt)=l_x-\max(n,l_x)\leq l_x-l_y=\delta\leq R$.
Furthermore, the inequality 
$$2n+R\geq d(z,x)+d(x,y)+d(y,z)\geq M> 4D+R$$
implies $n>2D$. Because $\triangle$ is $D$-slim and 
$$\operatorname{dist}(z,\overline{xy})\geq d(z,x)-d(x,y)\geq n-D>D$$
there must be a point $v=\gamma_y(t)$ on the geodesic $\gamma_y$ such that $d(z,v)\leq D$.
\begin{description}
\item[Case 1: $t\leq n$.] This means that $v$ lies between $w$ and $y$ on $\gamma_y$.  The two triangle inequalities
\begin{align*}
M=d(z,w)&\leq d(z,v)+d(v,w)\leq D+(n-t)\,,
\\l_y-t=d(pt,v)&\leq d(pt,z)+d(z,v)\leq (l_x-n)+D
\end{align*}
imply $M\leq 2D+l_x-l_y\leq 2D+R$, contradicting our assumption.
\item[Case 2: $t\geq n$.] This time $v$ lies between $w$ and $pt$ on $\gamma_y$ and we consider the triangle inequalities
\begin{align*}
M=d(z,w)&\leq d(z,v)+d(v,w)\leq D+(t-n)\,,
\\l_x-n=d(pt,z)&\leq d(pt,v)+d(z,v)\leq (l_y-t)+D
\end{align*}
which imply $M\leq l_y-l_x+D\leq D$. Again, there is a contradiction.
\end{description}
For fixed $x$ we clearly have $d(H(x,n),H(x,m))\leq|n-m|$ for all $n,m\in\N$. Thus, for arbitrary $(x,m),(y,m)\in X\times\N$ of distance at most $R$ the equality 
$$d(H(x,m),H(y,m))\leq 4D+2R$$
holds. Therefore $H$ is a coarse contraction to the point $pt$.
\end{example}

\begin{thm}\label{coarsesigmacontractibility}
Let $Z$ be a countably generated discrete coarse space which is coarsely contractible. Then the Rips complex $\mathcal{P}_Z$ is $\sigma$-contractible.
\end{thm}

\begin{proof}
Let $H:Z\times\N\to Z$ be a coarse contraction of $Z$ to a point $pt\in Z$. For every $n\in\N$ we obtain a continuous map
$$H_n:\mathcal{P}_Z\to\mathcal{P}_Z$$
by pushing forward probability measures along the maps $Z\to Z,\, z\mapsto H(z,n)$.

A $\sigma$-contraction $H_*:\,\mathcal{P}_Z\times[0,1]\to\mathcal{P}_Z$ of the Rips complex to the same point $pt$ is then given on the strips $\mathcal{P}_Z\times[1-2^{-n},1-2^{-n-1}]$ by interpolating linearly between
$$H_*(\mu,1-2^{-n})=H_n\quad\text{and}\quad H_*(\mu,1-2^{-n-1})=H_{n+1}$$
and defining $H_*|_{\mathcal{P}_Z\times\{1\}}:=pt$.

Obviously, $H_*$ restricts to the identity on $\mathcal{P}_Z\times\{0\}$ and $H_*(\mathcal{P}_Z\times\{1\}\cup\{pt\}\times[0,1])=\{pt\}$.

Continuity of $H_*$ at $\mathcal{P}_Z\times\{1\}$ follows from continuity of the restrictions $H_*|_{\Delta\times(1-\delta_\Delta,1]}$ for every simplex $\Delta$ of $\mathcal{P}_Z$ and some $\delta_\Delta>0$. But this is trivial, because the  third condition in the definition of coarse contractibility ensures that these restrictions are constantly equal to $pt$ for $\delta_\Delta>0$ small enough.

It remains to show that for each $n$ there is $m\geq n$ such that $H_*(\mathcal{P}_n\times[0,1])\subset\mathcal{P}_m$. Denote the generating entourages of the coarse structure on $Z$ as usually by $E_n$. The first property of coarse contractibility ensures that the entourage
$$\{((x,k),\,(y,l))\in(Z\times\N)^2\,|\,(x,y)\in E_n, |k-l|\leq 1\}$$
of the product coarse structure on $Z\times \N$ is mapped to some entourage $E_m$ of $Z$ by $H\times H$. If $\mu\in\mathcal{P}_n$, i.\,e.\ $\mu$ is a probability measure on $Z$ with $\supp\mu\times\supp\mu\subset E_n$, then for all $s\in[0,1]$ and $j\in\N$ the probability measure $\tilde\mu:=(1-s)\cdot H_j(\mu)+s\cdot H_{j+1}(\mu)$ satisfies
$$\supp\tilde\mu\subset\supp H_j(\mu)\cup\supp H_{j+1}(\mu)=H(\supp\mu\times\{j\}\cup\supp\mu\times\{j+1\})$$
and therefore we have
$\supp\tilde\mu\times\supp\tilde\mu\subset E_m,$
i.\,e.\ $\tilde\mu\in P_m$. This proves $H_*(\mathcal{P}_n\times[0,1])\subset\mathcal{P}_m$. 
\end{proof}

We shall now summarize our results to obtain the main result of this paper:
\begin{thm}\label{mainthmcoeff}
Let $X$ be a coarsely contractible countably generated coarse space. Then the unreduced coarse co-assembly map with coefficients in $D$,
$$\mu^*:K_{*}(\cfrak(X;D))\to KX^{1-*}(X\setminus\{pt\};D)\,,$$
is multiplicative.
\end{thm}
Of course, the secondary product on $KX^{1-*}(X\setminus\{pt\};D)$ is obtained by applying Definition \ref{secondaryproductdefinition} to the $\sigma$-contraction of the Rips complex obtained as in Theorem \ref{coarsesigmacontractibility} from the given coarse contraction.
With these ingredients, the theorem follows trivially from Theorem \ref{maintheorem}.

For $D=\C$ we obtain a ring homomorphism:
\begin{thm}\label{mainthm}
Let $X$ be a coarsely contractible countably generated coarse space. Then the unreduced coarse co-assembly map 
$$\mu^*:K_{*}(\cfrak(X))\to KX^{1-*}(X\setminus\{pt\})$$
is a ring homomorphism.
\end{thm}


\appendix
\section{Fréchet algebras}\label{sec:Frechet}
$\sigma$-$C^*$-algebras are a special case of Fréchet algebras:
\begin{mdef}[{c.f. \cite[Section 1]{PhiFre}} and the reference therein]
A locally multiplicatively convex Fréchet algebra (over the complex numbers) is a complex topological algebra such that its topology
\begin{itemize}\itemsep0pt\parskip0pt 
\item is Hausdorff,
\item is generated by a countable family of submultiplicative semi-norms $(p_n)_{n\in\N}$, 
\item is complete with respect to the family of semi-norms.
\end{itemize}
In the following, even without explicit mention, all Fréchet algebras appearing are assumed to be locally multiplicatively convex.
Furthermore, we will always assume that there are constants $c_n$, such that $p_{n}\leq c_np_{n+1}$ for all $n$.
\end{mdef}
Important examples of Fréchet algebras are the following algebras of smooth functions \cite[Section 1.1]{CunBiv}:

Let $C^\infty[a,b]$ denote the algebra of smooth functions $f:[a,b]\to\C$, whose derivations vanish at the endpoints. Furthermore, let $C_0^\infty(a,b]$, $C_0^\infty[b,a)$, $C_0^\infty(a,b)$ denote the subalgebras of $C^\infty[a,b]$ consisting of those functions which vanish at $a$ resp.\  $b$ resp.\ $a$ and $b$. All of them are Fréchet algebras with topology generated by the submultiplicative norms
$$p_n(f)=\|f\|+\|f'\|+\frac12\|f''\|+\dots+\frac1{n!}\|f^{(n)}\|.$$
They are the ingredients of a short exact sequence\footnote{Again, a sequence $0\to I\xrightarrow{\alpha}A\xrightarrow{\beta}B\to 0$ of Fréchet-algebras is called \emph{exact} if it is algebraically exact, $\alpha$ is a homeomorphism onto its image, and $\beta$ defines a homeomorphism $A/\ker(\beta)\to B$.}
 of Fréchet algebras
\begin{equation}\label{smoothexact}
0\to C_0^\infty(0,1)\to C_0^\infty[0,1)\to\C\to 0
\end{equation}
which splits by a continuous linear map $\C\to C_0^\infty[0,1)$.

We will make use of the external product in $kk_*$-theory. It is formulated in terms of the projective tensor product of Fréchet algebras:
\begin{mdef}[{cf.\ \cite[Section 43]{Treves},\cite[Theorem 2.3.(2e)]{PhiFre} and \cite[Section 1]{CunBiv}}]
Let $A,B$ be  two Fréchet algebras $A,B$ whose topologies are generated by the seminorms $(p_n),(q_n)$.
Their projective tensor product $A\otimes_\pi B$ is the completion of the algebraic tensor product with respect to the seminorms 
$$p\otimes_\pi q(z)=\inf\left\{\sum_{i=1}^mp(a_i)q(b_i):\,z=\sum_{i=1}^ma_i\otimes b_i,\,a_i\in A,b_i\in B\right\},$$
where $p$ is a continuous seminorm on $A$ and $q$ is a continuous seminorm on $B$. 

The projective tensor product of Fréchet algebras is again a Fréchet algebra, its topology being generated by the seminorms $p_n\otimes_\pi q_n$.
\end{mdef}
Its universal property is of course the following:
\begin{lem}\label{projtensoruniv}
Let $A,B,C$ be Fréchet algebras and let $f:A\times B\to C$ be a bilinear map such that for all continuous seminorms $r$ on $C$ there are continuous seminorms $p$ on $A$ and $q$ on $B$ and $K>0$ satisfying $r(f(a,b))\leq p(a)q(b)$ for all $a\in A$ and $b\in B$. Then there is a unique continuous  linear map $g:A\otimes_\pi B\to C$ such that $g(a\otimes b)=f(a,b)$ for all $a\in A,b\in B$.

If $g$ restricted to $A\odot B$ is an algebra homomorphism, then so is $g$.
\end{lem}
The projective tensor product with nuclear Fréchet algebras is exact \cite[Theorem 2.3.(3b)]{PhiFre}. So is  the projective tensor product of any Fréchet algebra with short exact sequences of Fréchet algebras, which split by a continuous linear map \cite[Remark following Definition 3.8]{CunBiv}. This is, because the projective tensor product is in fact a tensor product of Fréchet spaces. In particular, the projective tensor product of \eqref{smoothexact} with another Fréchet algebra $A$,
\begin{equation}\label{smoothtensorAexact}
0\to C_0^\infty(0,1)\otimes_\pi A\to C_0^\infty[0,1)\otimes_\pi A\to A\to 0,
\end{equation}
is exact.

For $\sigma$-$C^*$-algebras $A,B$, the universal property of Lemma \ref{projtensoruniv}
 yields a continuous homomorphism
\begin{equation}
A\otimes_\pi B\to A\otimes B,
\end{equation}
which will enables us to pass to the maximal tensor product of $\sigma$-$C^*$-algebras whenever an exact tensor product is needed.


\section{Cuntz's bivariant theory}\label{sec:kkTheory}
The bivariant groups $kk_*(A,B)$ ($*=0,1$) of \cite{CunBiv} are defined for all locally multiplicatively convex topological algebras $A,B$, but we are only interested in the special case of Fréchet algebras here. We mention some of the properties:
\begin{enumerate}
\item There is a composition product
$$\cdot:\,kk_i(A,B)\otimes kk_j(B,C) \to kk_{i+j}(A,C)$$
and a graded commutative exterior product
$$\otimes:\,kk_i(A_1,B_1)\otimes kk_j(A_2,B_2)\to kk_{i+j}(A_1\otimes_\pi A_2,B_1\otimes_\pi B_2)$$
whose compatibility properties are completely analogous to those of Kasparov's $KK$-theory. 
Some of them, but not all, are mentioned and proved in \cite{CunBiv}. The proof of the remaining properties is straightforward with the methods introduced therein.
\item Continuous homomorphisms $\alpha:A\to B$ induce elements $kk(\alpha)\in kk_0(A,B)$. For $\alpha:A\to B$ and $\beta:B\to C$ we have $kk(\beta\circ\alpha)=kk(\alpha)\cdot kk(\beta)$. The groups $kk_*$ are functorial in the second and contravariantly functorial in the first variable with respect to continuous homomorphisms. This functoriality is given by the composition product with the elements induced by the homomorphisms.
\item The covariant functor $kk_*(\C,\_)$ (defined on the category of Fréchet algebras) is naturally isomorphic to the $K$-theory functor $K_*$  of Phillips.
\item For each $A$, there is an invertible element $\mathfrak{b}_A\in kk_1(A,C_0^\infty(0,1)\otimes_\pi A)$. In fact, it is enough to know $\mathfrak{b}=\mathfrak{b}_\C\in K_1(C_0^\infty(0,1))$, because $\mathfrak{b}_A=\mathfrak{b}\otimes kk(\id_A)$. 
The composition product 
$$\_\cdot \mathfrak{b}_B:\,kk_i(A,B)\to kk_{i+1}(A,C_0^\infty(0,1)\otimes_\pi B)$$
is an isomorphism, which coincides up to a sign with the exterior product
$$\mathfrak{b}\otimes \_:\,kk_i(A,B)\to kk_{i+1}(A,C_0^\infty(0,1)\otimes_\pi B).$$
We will always use the latter as the Bott periodicity identification $kk_i(A,B)\cong kk_{i+1}(A,C_0^\infty(0,1)\otimes_\pi B)$, which becomes 
\begin{equation}\label{BottPeriodicity}
\mathfrak{b}\otimes\_:\,K_i(A)\cong K_{i+1}(C_0^\infty(0,1)\otimes_\pi A)
\end{equation}
in the special case of $K$-theory of Fréchet algebras.
\item There are the usual six-term exact sequences
\begin{equation}\label{sixtermkk}
\xymatrix{
kk_0(D,I)\ar[r]&kk_0(D,A)\ar[r]&kk_0(D,B)\ar[d]
\\kk_1(D,B)\ar[u]&kk_1(D,A)\ar[l]&kk_1(D,I)\ar[l]}
\end{equation}
(and analogously in the first variable) 
 for all short exact sequences 
\begin{equation}\label{splitshortexactsequence}
0\to I\to A\xrightarrow{q} B\to 0
\end{equation}
of Fréchet algebras, which split by a continuous linear map. 
We briefly recall the construction of the connecting homomorphisms. 
Even if the sequence does not split linearly, there is an exact sequence
\begin{align*}
\dots &\to kk_*(D,C_0^\infty(0,1)\otimes_\pi A)\to kk_*(D,C_0^\infty(0,1)\otimes_\pi B)\to
\\&\to kk_*(D,C_q^\infty)\to kk_*(D,A)\to kk_*(D,B),
\end{align*}
where 
$C_q^{\infty}=\{(x,f)\in A\oplus C_0^\infty[0,1)\otimes_\pi B:\,q(x)=f(0)\}$ 
(the evaluation $f(0)$ is given by the map in \eqref{smoothtensorAexact}) 
is the smooth mapping cone. 

We can allways replace $kk_*(D,C_0^\infty(0,1)\otimes_\pi B)$ by $kk_{*-1}(D,B)$ using Bott periodicity. 
The exact sequence
\begin{equation}\label{idealincone}
0\to I\to C_q^{\infty}\to C_0^\infty[0,1)\otimes_\pi B\to 0
\end{equation}
splits, if \eqref{splitshortexactsequence} splits, and $C_0^\infty[0,1)\otimes_\pi B$ is diffeotopically contractible. 
Thus, the exact sequence associated to \eqref{idealincone} shows, that
$$kk_*(D,I)\to kk_*(D,C_q^{\infty})$$
is an isomorphism in this case. We now obtain \eqref{sixtermkk} from these pieces and the boundary map is therefore given by the composition
$$kk_{*-1}(D,B)\xrightarrow{\mathfrak{b}\otimes\_}kk_*(D,C_0^\infty(0,1)\otimes_\pi B)\to kk_*(D,C_q^{\infty})\xleftarrow{\cong}kk_*(D,I).$$

However, if we work with $K$-theory of Fréchet algebras, i.e.\ $D=\C$, then there is always a six term exact sequence, even if \eqref{splitshortexactsequence} does not split. Thus, $K_*(I)\to K_*(C_q^{\infty})$ is always an isomorphism by diffeotopy invariance of $K$-theory. So we see that the boundary map of the exact sequence 
$$\xymatrix{
K_0(I)\ar[r]&K_0(A)\ar[r]&K_0(B)\ar[d]
\\K_1(B)\ar[u]&K_1(A)\ar[l]&K_1(I)\ar[l]
}$$
can be chosen to be
$$K_{*-1}(B)\xrightarrow{\mathfrak{b}\otimes\_}K_*(C_0^\infty(0,1)\otimes_\pi B)\to K_*(C_q^{\infty})\xleftarrow{\cong}K_*(I).$$
\end{enumerate}

Now that we have reviewed some basic properties of $K$-theory of Fréchet algebras, we specialize to $\sigma$-$C^*$-algebras and prove that the remaining Axioms given in Section \ref{sec:KTheory} are satisfied.
We start with Axiom \ref{Ktheoryaxiomproduct}.
\begin{mdef}
The exterior product in $K$-theory of $\sigma$-$C^*$-algebras with respect to the maximal tensor product is defined to be the composition
$$K_i(A)\otimes K_j(B)\to K_{i+j}(A\otimes_\pi B)\to K_{i+j}(A\otimes B).$$
\end{mdef}
Associativity follows from commutativity of 
$$\xymatrix{
A\otimes_\pi B\otimes_\pi C\ar[r]\ar[d]&A\otimes_\pi(B\otimes C)\ar[d]
\\(A\otimes B)\otimes_\pi C\ar[r]&A\otimes B\otimes C.
}$$
The product is obviously graded commutative and natural in both $A$ and $B$, because it has these properties with respect to the projective tensor product.

Given a short exact sequence of $\sigma$-$C^*$-algebras
$$0\to I\to A\xrightarrow{q} B\to 0,$$
we can also define the continuous mapping cone 
$$C_q=\{(x,f)\in A\oplus C_0[0,1)\otimes B:\,q(x)=f(0)\}.$$
It contains the ideals $C_0(0,1)\otimes B$ and $I$. As $C_q/I=C_0[0,1)$ is contractible, the inclusion $I\subset C_q$ induces an isomorphism
$K_*(I)\cong K_*(C_q),$ too.

The canonical continuous homomorphisms $\iota:C_0^\infty(0,1)\to C_0(0,1)$, $C_0^\infty(0,1)\otimes_\pi B\to C_0(0,1)\otimes_\pi B\to C_0(0,1)\otimes B$ and $C_q^\infty\to C_q$ induce a commutative diagram
$$\xymatrix{
K_{*-1}(B)\ar[rr]^{\mathfrak{b}\otimes\_}\ar@{=}[d]&&K_*(C_0^\infty(0,1)\otimes_\pi B)\ar[r]\ar[d]& K_*(C_q^{\infty})\ar[dd]&\ar[l]_{\cong}K_*(I)\ar@{=}[dd]
\\K_{*-1}(B)\ar[rr]^{\iota_*(\mathfrak{b})\otimes\_}\ar@{=}[d]&&K_*(C_0(0,1)\otimes_\pi B)\ar[d]&&
\\K_{*-1}(B)\ar[rr]^{\iota_*(\mathfrak{b})\otimes\_}&&K_*(C_0(0,1)\otimes B)\ar[r]& K_*(C_q)&\ar[l]_{\cong}K_*(I)
}$$
If we use the letter $\mathfrak{b}$ for $\iota_*(\mathfrak{b})$, too, we obtain the $\sigma$-$C^*$-algebra description of the connecting homomorphism:
\begin{lem}
The connecting homomorphism in $K$-theory of $\sigma$-$C^*$-algebras associated to the short exact sequence 
$$0\to I\to A\xrightarrow{q} B\to 0$$
is the composition
$$K_{*-1}(B)\xrightarrow{\mathfrak{b}\otimes\_}K_*(C_0(0,1)\otimes B)\to K_*(C_q)\xleftarrow{\cong}K_*(I).$$
\end{lem}

Proving compatibility of boundary maps and exterior products as mentioned in Axiom \ref{Ktheorypropertyboundary} is now straightforward:
Consider the diagram
{\small
$$\xymatrix@C=0pc{
K_1(C_0(0,1))\otimes K_i(B)\otimes K_j(D)\ar[d]\ar[rr]&&K_1(C_0(0,1))\otimes K_{i+j}(B\otimes D)\ar[d]
\\K_{i+1}(C_0(0,1)\otimes B)\otimes K_j(D)\ar[rr]\ar[d]&&K_{i+j+1}(C_0(0,1)\otimes B\otimes D)\ar[d]\ar[dl]
\\K_{i+1}(C_q)\otimes K_j(D)\ar[r]&K_{i+1}(C_q\otimes D)\ar[r]&K_{i+j+1}(C_{q\otimes\id_D})
\\K_{i+1}(I)\otimes K_j(D)\ar[u]^{\cong}\ar[rr]&&K_{i+j+1}(I\otimes D).\ar[u]^\cong\ar[ul]
}$$}
The upper square commutes by associativity of the external product, the other two quadrilaterals commute by functoriality of the external product under the canonical inclusions $C_0(0,1)\otimes B\subset C_q$ and $I\subset C_q$ and the triangles commute already on the level of homomorphisms.

This proves commutativity of
$$\xymatrix{
K_i(B)\otimes K_j(D)\ar[r]\ar[d]& K_{i+1}(I)\otimes K_j(D)\ar[d]
\\K_{i+j}(B\otimes D)\ar[r]&K_{i+j+1}(I\otimes D).
}$$
Commutativity of 
$$\xymatrix{
K_i(D)\otimes K_j(B)\ar[r]\ar[d]& K_{i}(D)\otimes K_{j+1}(I)\ar[d]
\\K_{i+j}(D\otimes B)\ar[r]&K_{i+j+1}(D\otimes I)
}$$
up to a sign $(-1)^i$ follows from the commutativity of the first  by graded commutativity of the external product.


{
\bigskip

\noindent
\textsc{Institut f\"ur Mathematik, Universit\"at Augsburg, D-86135 Augsburg, Germany }

\noindent
\textit{E-mail address:} \textsf{christopher.wulff@math.uni-augsburg.de}

}


\begin{thebibliography}{{Cun}97}

\bibitem[BH99]{BriHae}
Martin~R. Bridson and Andr{\'e} Haefliger.
\newblock {\em Metric spaces of non-positive curvature}, volume 319 of {\em
  Grundlehren der Mathematischen Wissenschaften [Fundamental Principles of
  Mathematical Sciences]}.
\newblock Springer-Verlag, Berlin, 1999.

\bibitem[{Cun}97]{CunBiv}
Joachim {Cuntz}.
\newblock {Bivariante $K$-Theorie f\"ur lokalconvexe Algebren und der
  Chern-Connes-Charakter.}
\newblock {\em {Doc. Math., J. DMV}}, 2:139--182, 1997.

\bibitem[EM06]{EmeMey}
Heath Emerson and Ralf Meyer.
\newblock Dualizing the coarse assembly map.
\newblock {\em J. Inst. Math. Jussieu}, 5(2):161--186, 2006.

\bibitem[Gro87]{GroHyperbolic}
M.~Gromov.
\newblock Hyperbolic groups.
\newblock In {\em Essays in group theory}, volume~8 of {\em Math. Sci. Res.
  Inst. Publ.}, pages 75--263. Springer, New York, 1987.

\bibitem[HG04]{HigGue}
Nigel Higson and Erik Guentner.
\newblock Group {$C^\ast$}-algebras and {$K$}-theory.
\newblock In {\em Noncommutative geometry}, volume 1831 of {\em Lecture Notes
  in Math.}, pages 137--251. Springer, Berlin, 2004.

\bibitem[HR00]{HigRoe}
Nigel Higson and John Roe.
\newblock {\em Analytic {$K$}-homology}.
\newblock Oxford Mathematical Monographs. Oxford University Press, Oxford,
  2000.
\newblock Oxford Science Publications.

\bibitem[Phi88]{PhiInv}
N.~Christopher Phillips.
\newblock Inverse limits of {$C^*$}-algebras.
\newblock {\em J. Operator Theory}, 19(1):159--195, 1988.

\bibitem[{Phi}89]{PhiRep}
N.Christopher {Phillips}.
\newblock {Representable K-theory for $\sigma$-C${}\sp*$-algebras.}
\newblock {\em {$K$-Theory}}, 3(5):441--478, 1989.

\bibitem[{Phi}91]{PhiFre}
N.Christopher {Phillips}.
\newblock {$K$-theory for Fr\'echet algebras.}
\newblock {\em {Int. J. Math.}}, 2(1):77--129, 1991.

\bibitem[Roe95]{RoeFoliations}
John Roe.
\newblock From foliations to coarse geometry and back.
\newblock In {\em Analysis and geometry in foliated manifolds ({S}antiago de
  {C}ompostela, 1994)}, pages 195--205. World Sci. Publ., River Edge, NJ, 1995.

\bibitem[Roe96]{RoeITCGTM}
John Roe.
\newblock {\em Index theory, coarse geometry, and topology of manifolds},
  volume~90 of {\em CBMS Regional Conference Series in Mathematics}.
\newblock Published for the Conference Board of the Mathematical Sciences,
  Washington, DC; by the American Mathematical Society, Providence, RI, 1996.

\bibitem[Roe03]{RoeCoarseGeometry}
John Roe.
\newblock {\em Lectures on coarse geometry}, volume~31 of {\em University
  Lecture Series}.
\newblock American Mathematical Society, Providence, RI, 2003.

\bibitem[Tr{\`e}67]{Treves}
Fran{\c{c}}ois Tr{\`e}ves.
\newblock {\em Topological vector spaces, distributions and kernels}.
\newblock Academic Press, New York-London, 1967.

\end{thebibliography}
\end{document}